\documentclass[12pt,leqno,twoside]{article}
\usepackage{amssymb}
\usepackage{amsmath}
\usepackage{amsthm,amscd}
\usepackage{bbm}
\usepackage{bm}
\usepackage{a4,indentfirst,latexsym}
\usepackage{graphics}
\usepackage{mathrsfs}
\usepackage{cite,enumitem,graphicx}
\usepackage[colorlinks=true,urlcolor=blue,
citecolor=blue,linkcolor=blue,linktocpage,pdfpagelabels,
bookmarksnumbered,bookmarksopen]{hyperref}
\usepackage[colorinlistoftodos]{todonotes}
\usepackage{color}

\allowdisplaybreaks 
\newtheorem{theorem}{Theorem}[section]
\newtheorem{lemma}{Lemma}[section]

\newtheorem{definition}{\hspace{2em} Definition}[section]
\newtheorem{corollary}{Corollary}[section]

\newtheorem{remark}{\it Remark}[section]

\newtheorem{claim}{\it {Claim}}[section]

\def\B{\mathbb{B}}

\def\N{\mathbb{N}}

\def\R{\mathbb{R}}

\def\cG{\mathcal{G}}

\def\cI{\mathcal{I}}

\def\cL{\mathcal{L}}
\def\cM{\mathcal{M}}

\def\cO{\mathcal{O}}
\def\cP{\mathcal{P}}

\def\cR{\mathcal{R}}
\def\cS{\mathcal{S}}
\def\cT{\mathcal{T}}

\newcommand{\la}{\langle}
\newcommand{\ra}{\rangle}
\newcommand{\ran}{\rangle}
\newcommand{\lan}{\langle}

\newcommand{\intsigma}{\mathring\Sigma}

\def\d{\mathrm{d}}

\newcommand{\lam}{\lambda}
\newcommand{\al}{\alpha}
\newcommand{\gam}{\gamma}
\newcommand{\cc}{\mathrm{c}}

\newcommand{\eps}{\varepsilon}
\newcommand{\bdelta}{\boldsymbol{\delta}}
\newcommand{\oH}{\overline{\mathrm{H}}^1}
\newcommand{\ii}{\mathrm{i}}
\newcommand{\ff}{\mathbf{F}^{K, Q_1}_{p,q}}

\newcommand{\idelta}{\delta^{\gam_i}_{\lam_i,\xi_i}}

\newcommand{\rH}{\mathrm{H}}
\newcommand{\rL}{\mathrm{L}}

\newcommand\mytop[2]{\genfrac{}{}{0pt}{}{#1}{#2}}

\numberwithin{equation}{section}

 \def\XXint#1#2#3{{\setbox0=\hbox{$#1{#2#3}{\int}$ } \vcenter{\hbox{$#2#3$ }}\kern-.6\wd0}}

\def\Dc{\mathscr{D}}

\begin{document}
	\title{Blow-up solutions for mean field equations with non-quantized singularities on Riemann surfaces with boundary}
	\author{\medskip  {\sc Mohameden Ahmedou}, \ \
		{\sc Zhengni Hu}, \ \ 
		{\sc Miaomiao Zhu }
	}

	\date{\today}
	\maketitle

	\begin{abstract}
		We study mean field equations with singular sources on a compact Riemann surface with boundary $(\Sigma,g)$, subject to homogeneous Neumann boundary conditions:
		\[
		\begin{cases}
			-\Delta_g v
			= \rho\!\left( \dfrac{V e^{v}}{\int_\Sigma V e^{v}\,\d v_g}
			- \dfrac{1}{|\Sigma|_g}\right)
			- \displaystyle\sum_{\xi\in Q} \dfrac{\varrho(\xi)}{2}\,\gam(\xi)
			\left(\bm{\delta}_{\xi}- \dfrac{1}{|\Sigma|_g}\right)
			& \text{in }\Sigma,\\[6pt]
			\partial_{\nu_g} v = 0
			& \text{on }\partial\Sigma.
		\end{cases}
		\]
		Here, $V$ is a smooth positive function, $\rho$ is a non-negative parameter,  $Q\subset\Sigma$ is a finite set of prescribed singular points, and the singular weights satisfy
		$\gam(\xi)\in(-1,+\infty)\setminus(\mathbb{N}\cup\{0\})$.
		The coefficients are given by $\varrho(\xi)=8\pi$ for $\xi \in\Sigma\setminus\partial\Sigma$
		and $\varrho(\xi)=4\pi$ for $\xi \in\partial\Sigma$.
		
		We construct blow-up solutions in the non-quantized singular regime, including purely singular and mixed singular-regular blow-up cases, with parameters approaching resonant values.
		The construction is achieved via a Lyapunov-Schmidt reduction under suitable stability assumptions.

		
		\hspace{2em}
		
		
	\end{abstract}

	
	\section{Introduction}
	Let $(\Sigma, g)$ be a compact Riemann surface with the interior $\intsigma$ and smooth boundary $\partial\Sigma$. For simplicity, we assume the area of $\Sigma$ is unit, i.e. $|\Sigma|_g=1$. We consider the following  mean field equations with singular sources: 
	\begin{equation}\label{eq:s_mf}
		\left\{
		\begin{aligned}
			-\Delta_g v
			&=\rho\left(\frac{V e^{v}}{\int_{\Sigma} V e^{v}\, \mathrm{d}v_g}-1\right)
			-\sum_{\xi\in Q}\frac{\varrho(\xi)}{2}\,\gam(\xi)(\bm{\delta}_{\xi}-1)
			&&\text{in }\Sigma,\\
			\partial_{\nu_g} v&=0 &&\text{on }\partial\Sigma.
		\end{aligned}
		\right.
	\end{equation}
	where $\Delta_g$ is the Laplace-Beltrami operator, $\d v_g$ is the area element of $(\Sigma,g)$,  $\nu_g$ denotes the outward unit normal vector along $\partial \Sigma$,  $V$ is a smooth positive function,  $\rho$ is a non-negative parameter,  $Q$ is a given finite subset of $\Sigma$,  $\gam(\xi)\in (0, +\infty)\setminus(\N\cup\{ 0\})$ for all $\xi\in Q$, and  $\bm{\delta}_\xi$ is the Dirac measure on $\Sigma$ concentrated at $\xi$. The coefficients  $\varrho(\xi)=
	8\pi $  if $\xi\in \intsigma$ and $
	4\pi $ if $\xi\in \partial\Sigma. $


	
	Singular mean field equations arise naturally in several areas of physics and mathematics. They appear as mean field limits of two-dimensional turbulent Euler flows \cite{Caglioti1992, Caglioti1995, Chanillo1994, Kiessling1993}, and as fundamental models for self-dual vortices in the Electroweak and Chern-Simons gauge theories \cite{AmbjornOlesen1988, SpruckYang1992, Tarantello1996, Yang2001}. In conformal geometry, they are closely related to the problem of prescribing Gaussian curvature on surfaces, both with conical singularities \cite{Troyanov1991} and in the regular case \cite{KazdanWarner1974, ChangYang1987, ChangYang1988}. A vast literature is devoted to the mean field equation \eqref{eq:mf}, addressing existence, uniqueness, and blow-up behavior; see, for instance, \cite{Bartolucci2013, Bartolucci2019, BartolucciDeMarchis2015, BartolucciDeMarchisMalchiodi2011, BartolucciChenLinTarantello2004, Bartolucci2018, BartolucciLin2014, BartolucciMalchiodi2013, BartolucciTarantello2002, BartolucciTarantello2017, FangLai2016, WeiZhang2018, Zhang2009, CCL2015} and the references therein.
	
	
	In the regular case $Q=\emptyset$, the asymptotic behavior and existence of blow-up solutions have been extensively studied; see, for example, \cite{brezis_uniform_1991, CL2002, Li1999, ma_convergence_2001, Nagasaki1990AsymptoticAF, PT2001, EF2014, Esposito2005, del_pino_singular_2005, ABH2024KS}.   
	For the singular problem, a blow-up point $\xi$ can be divided into three types: 
	regular (i.e., $\xi\notin Q$), non-quantized (i.e., $\gam(\xi)\notin \mathbb{N}_0:= \mathbb{N}\cup \{ 0\}$), 
	and quantized (i.e., $\gam(\xi)\in \mathbb{N} $).   The analysis is substantially more delicate in this setting. The asymptotic behavior near regular and non-quantized singularities has been studied in \cite{BartolucciChenLinTarantello2004, Bartolucci2002Liouville, CCL2015, Zhang2009}, where blow-up solutions exhibit single-bubble profiles after rescaling. Concerning existence results in the non-quantized singular case, Esposito \cite{Esposito2005SIAM} constructed blow-up solutions for Liouville-type equations with Dirichlet boundary conditions and positive singular weights. At quantized singularities, the spherical Harnack inequality may fail, making the problem more challenging; nevertheless, substantial progress has been achieved in recent years \cite{Wu2023, WeiZhang2022, WeiZhang2021, KuoLin2016, DAprileWei2020, BartolucciTarantello2017, DA-Wei-Zhang2024}.

	Denote
	\begin{equation}
		\label{eq:hQ}
		h_Q:=\sum_{\xi \in Q} \frac{\varrho(\xi)}{2}\gam(\xi)\, G^g(\cdot, \xi).
	\end{equation}
	Here $G^g(\cdot,\cdot)$ denotes the Green function associated with the Neumann Laplacian; see Section \ref{sec:pre} for details.
	It follows that $-\Delta_g h_Q= \sum_{\xi\in Q} \frac{\varrho(\xi)}{2}\gam(\xi) (\bm{\delta}_\xi -1)$.  
	Setting $u(x)=v(x)+h_Q(x)$ and $K(x)=V(x)\, e^{-h_Q(x)}$, we obtain that $u$ solves equation~\eqref{eq:s_mf}, where $K$ is a non-negative smooth function whose zero set is precisely $Q$.

	The  mean field equations with singular sources~\eqref{eq:s_mf} can be reformulated as follows: 
	\begin{equation}
		\label{eq:mf} 
		\left\{\begin{aligned}
			-\Delta_g u &=\rho  \Big( \frac{Ke^u}{\int_{\Sigma} Ke^ u \, \d v_g} -1\Big)  && \text{ in } \Sigma,\\
			\partial_{\nu_g} u&=0 && \text{ on } \partial\Sigma. 
		\end{aligned}\right.
	\end{equation}

	Since the solution space of \eqref{eq:mf} is invariant under adding  any constant, 
	we consider the solutions in a subspace of $H^1(\Sigma)$ with average $0$, 
	$ \oH:= \{u\in H^1(\Sigma): \int_{\Sigma} u \, \d v_g =0\}. $
	The corresponding Euler-Lagrange functional is given by:
	\[ J_{\rho}(u)= \frac 1 2 \int_{\Sigma} |\nabla u|^2_g \, \d v_g -\rho  \ln \int_{\Sigma} K e^u \, \d v_g, \text{ for all } u\in \oH. \]


	It is well known that blow-up phenomena can occur only when the parameter $\rho$ belongs to the so-called set of resonant   values. Given a  singular set $Q \subset \Sigma$ with $|Q| < \infty$, the set of resonant values is defined as  
	\begin{equation}
		\label{eq:def_critical_g}
		\mathcal{G} := \Big\{ 4\pi(2p + q) + \sum_{\xi  \in Q_1} (1 + \gam(\xi ))  \varrho(\xi )> 0 : Q_1 \subset Q,\; p, q \in \mathbb{N} \cup \{0\} \Big\},
	\end{equation}
	as established in mean field equations by blow-up analysis (see~\cite{Bartolucci2002Liouville} for closed surfaces, and with minor modifications for Riemann surfaces with boundary).
	
	
	It is natural to ask whether blow-up phenomena occur for all $\rho_*\in\mathcal{G}$. 
	In this paper, we establish sufficient conditions for the existence of sequences of blow-up solutions to \eqref{eq:mf} as the parameter goes to $\rho_*\in \cG$. 
	While the case where blow-up occurs only at regular points was studied in \cite{ABH2024KS} under a $C^1$-stability assumption on a reduced function, we construct  purely singular and mixed regular-singular blow-up solutions on Riemann surfaces with boundary in the non-quantized regime, including the presence of negative singular weights, by a Lyapunov-Schmidt reduction. 
	The case of quantized singularities, which may involve non-simple blow-up behavior, is left for future study.
	
	Although the Lyapunov-Schmidt reduction is a standard tool for constructing blow-up solutions, in the presence of singular sources additional difficulties arise from boundary blow-up points and non-quantized singularities, requiring the introduction of suitable projected bubbles and refined estimates.

	When blow-up occurs only at singular points, our first main result is stated below.
	
	\begin{theorem}\label{thm:main0}
		Let $Q_1\subset Q$ be a nonempty subset, and set $N:=|Q_1|>0$.  
		Then there exists a family of blow-up solutions $v_\eps$ to \eqref{eq:s_mf} with
		\[
		\rho_\eps \to \rho_* := \sum_{\xi\in Q_1} (1+\gam(\xi))\,\varrho(\xi)
		\quad \text{as } \eps \to 0,
		\]
		such that $v_\eps$ blows up precisely at the points in $Q_1$.

		Moreover, up to a subsequence, the following estimates hold:  	as $\eps\to 0$
		$$ \rho_\eps \frac{V e^{v_\eps} }{\int_\Sigma V e^{v_\eps} \, \d v_g}  \, \d v_g  \stackrel{*}{\rightharpoonup}   \sum_{\xi \in Q_1} (1+\gam(\xi ))\varrho(\xi) \bm{\delta}_{\xi} \qquad \text{(weak-$*$ convergence)}$$
		and  
		\[  J_{\rho^\eps}(u_\eps)\to \rho_* -\rho_* \ln  \Big(\frac {\rho_*}{8}\Big)+ 
		\sum_{\xi \in Q_1} 2 (1+\gam(\xi))\varrho(\xi) \ln (1+\gam(\xi)),  \] 
		where $u_\eps(x)= v_\eps(x) + h_Q(x)$ and $K(x) = V(x) e^{-h_Q(x)}$. 
	\end{theorem}

	Before stating the second main result, we introduce the corresponding reduced function and define the concept of $ C^1 $-``stable'' critical points.

	For  any fixed $Q$,  we rewrite the resonant value set \eqref{eq:def_critical_g} as follows: 
	\[ \cG:=\{ n_m: m\in\N\text{ with }  n_m-n_{m'}>0 \text{ for  any } m> m'\}. \]
	Given $m\in\N$, we set $\rho_*:=4\pi n_m$, $\cP(Q):=\{ B: B \subset Q\}$ is the power set of $Q,$ and  
	\[ \cI_m:=\Big\{ (p,q, Q_1)\in \N\times \N \times\cP(Q): 4\pi n_m= 4\pi(2p+q)+ \sum_{\xi \in Q_1} (1+\gam(\xi ))\varrho(\xi)  \Big\}.\]
	Denote  for any $(p,q, Q_1)\in \cI_{m}$ with $N=|Q_1|$, 
	$$\gam_j:=\begin{cases}
		0 & \text{ for } 1\leq j\leq  p+q,\\
		\gam(\xi^*_j) &\text{ for }p+q< j\leq  p+q+N.
	\end{cases}$$


	

	\begin{definition}  A critical point $x\in M$ of a $C^1$-function $F: M \to \R$, defined on a manifold $M$, is said to be $C^1$-stable if for every neighborhood $U$ of $x$ in $M$ there exists $\eps>0$ such that the following statement holds: 
		
		{	If $G: U \to \R$ is a $C^{1}$-function with $\|F-G\|_{C^1(U)}<\eps$, then G has a critical point in $U$.}
	\end{definition}
	
	In the case where the blow-up set consists of both singular and regular points, 
	the presence of regular blow-up points produces nontrivial kernel directions in the linearized operator, 
	so that it is no longer naturally invertible. 
	To handle the effect of the regular blow-up points, 
	we therefore impose a stability condition on the reduced functional  introduced in Section \ref{sec:pre}.  
	
	A sufficient condition for the existence of blow-up solutions to the singular mean field equation is given by the following result.
	\begin{theorem}
		\label{thm:main}
		Given $m\in \N$, for any $(p,q, Q_1)\in \cI_m$ with $N=|Q_1|>0$, suppose that the reduced function $\ff$ defined by \eqref{def:reduce_fun_0} has a $C^1$-stable critical point $\xi^{0,*}:=(\xi^*_1,\ldots,\xi^*_{p+q})$ and 
		$$ \gam_*:=\max \{\gam_i: i=1,\ldots, p+q+N\}<1 .$$
		Then, 
		there exists a sequence of blow-up solutions  $v_\eps$ of \eqref{eq:s_mf} with parameter $\rho_\eps\to \rho_*:= 4\pi n_m$  and $\xi^{0,\eps}:=(\xi^\eps_1,\ldots,\xi^\eps_{p+q})\to \xi^{0,*}$ as $\eps\to 0 $, 
		which blows up exactly at   $\xi_1^*, \ldots, \xi^*_{p+q+N}$ satisfying that  $Q_1=\{ \xi_{p+q+1}^*,\ldots,\xi^*_{p+q+N}\}$. 
		
		Moreover, up to a subsequence, the following estimates hold:  	as $\eps\to 0$
		\[  \rho_\eps \frac{V e^{v_\eps} }{\int_\Sigma V e^{v_\eps} \, \d v_g}  \, \d v_g  \stackrel{*}{\rightharpoonup}   \sum_{i=1}^{p+q+N} (1+\gam_i )\varrho(\xi^*_i) \bm{\delta}_{\xi^*_i};  \]
		and  
		\[  J_{\rho^\eps}(u_\eps)\to \rho_* -\rho_* \ln  \Big(\frac {\rho_*}{8}\Big)+ 
		\sum_{i=1}^{p+q+N} 2 (1+\gam_i)\varrho(\xi_i) \ln (1+\gam_i)-\frac 1 2 \ff(\xi^{0,*}),  \] 
		where $u_\eps(x)= v_\eps(x) + h_Q(x)$ and $K(x) = V(x) e^{-h_Q(x)}$. 
	\end{theorem}

	\begin{remark}
		For technical reasons, we restrict to the regime $\gam_*<1 $. 
		When $\gam_*>1$, the remainder term produced in the Lyapunov-Schmidt reduction decays too slowly with respect to the small parameter $\eps>0$. 
		As a consequence, the perturbative argument used to locate a critical point of the reduced functional fails. 
	\end{remark}

	We can extend the homogeneous Neumann boundary condition in \eqref{eq:s_mf} to be a non-homogeneous one, which is related to the prescribed   Gaussian curvature and boundary geodesic curvature problem. 
	Suppose that $(\Sigma, g)$ is a compact Riemann surface with smooth boundary, where $g$ is a smooth conformal metric. 
	Define  
	$ g_0:= \rho_*  \frac{  e^{-h_Q}e^w}{\int_{\Sigma} 2\tilde{K} e^{-h_Q} e^w \, \d v_g} g,$
	where $h_Q$ is defined by \eqref{eq:hQ}, $\rho_*>0$, and $\tilde{K}$ is a given  smooth, positive-valued  function on $\Sigma$. 
	
	Suppose that $g_0$ is a conformal metric with prescribed Gaussian curvature  $\tilde{K}$ and vanishing geodesic curvature on the boundary.
	Using the Gauss-Bonnet formula (see, for instance, \cite[Proposition 1]{Troyanov1991}), we have
	$ \frac 1 {2\pi}\int_{\Sigma} K_g  \, \d v_{g} + \frac 1 {2\pi}\int_{\partial\Sigma} k_g \, \d s_{g} =\bm{\chi}(\Sigma), $
	and 
	$  \frac 1 {2\pi}\int_{\Sigma} \tilde{K}\, \d v_{g_0}= \bm{\chi}(\Sigma)-\sum_{\xi\in Q\cap\intsigma} \gam(\xi)-\sum_{\xi \in Q\cap \partial\Sigma}\frac{\gam(\xi)}{2},$
	where $\bm{\chi}(\Sigma)$ is  the Euler characteristic of $\Sigma$, $K_g$ is the constant Gaussian curvature of $(\Sigma, g)$, $k_g$ is the geodesic curvature of $\partial\Sigma$ with respect to $g$ and $\d s_g$ is the line element of $\partial\Sigma$ induced by  $g$. It follows that 
	$ \rho_*= 4\pi (   \bm{\chi}(\Sigma) - \sum_{\xi \in Q\cap\intsigma}\gam(\xi)-\sum_{\xi \in Q\cap\partial\Sigma}\frac{\gam(\xi)}{2}). $
	Following the arguments in \cite{BartolucciDeMarchisMalchiodi2011,Troyanov1991}, 
	we obtain the following mean field type equation, 
	\begin{equation}\label{eq:Gauss_mf}
		\left\{  \begin{aligned}
			-\Delta_g w &= \rho \Big( \frac{2\tilde{K}  e^{-h_Q}e^w}{\int_{\Sigma}2 \tilde{K} e^{-h_Q} e^w \d v_g} -1 \Big)+ 2\int_{\partial\Sigma} k_g \, \d s_g && \text{ in } \Sigma,  \\
			\partial_{\nu_g} w&= - 2k_g && \text{ on }\partial\Sigma, 	\\
		\end{aligned}\right.
	\end{equation}
	where the parameter $\rho>0$. 
	
	We introduce a smooth auxiliary function $h$  solving 
	\begin{equation*}
		\left\{ \begin{aligned}
			-\Delta_g h& = 2\int_\Sigma k_g \,\d s_g && \text{ in }\Sigma, \\
			\partial_{\nu_g} h&=  -2k_g && \text{ on }\partial\Sigma, 
		\end{aligned}\right.
	\end{equation*} 
	and define $v:=w-h$ and $V:= 2 \tilde{K} e^{-h_Q+h}$. Then equation \eqref{eq:Gauss_mf} reduces to a mean field equation with homogeneous Neumann boundary condition of the form \eqref{eq:s_mf}.
	
	Consequently, based on Theorem \ref{thm:main}, we have the following corollary for the geometric boundary condition problem: 
	\begin{corollary}
		Given $m\in \N$, for any $(p,q, Q_1)\in \cI_m$ with $N=|Q_1|>0$, let $K= 2 \tilde{K} e^{-h_Q+h}$. 
		Suppose that the reduced function $ \ff$ defined by \eqref{def:reduce_fun_0} has a $C^1$-stable critical point $\xi^{0,*}:=(\xi^*_1,\ldots,\xi^*_{p+q})$ and 
		$$\gam_*:= \max\{ \gam_i: i=1,\ldots, p+q+N\}<1.$$ 
		Then, 
		there exists a sequence of blow-up solutions  $w_\eps$ of \eqref{eq:Gauss_mf} with parameter $\rho_\eps\to \rho_*:= 4\pi n_m$  and $\xi^{0,\eps}:=(\xi^\eps_1,\ldots,\xi^\eps_{p+q})\to \xi^{0,*}$ as $\eps\to 0 $, 
		which blows up at   $\xi_1^*, \ldots, \xi^*_{p+q+N}$ with $Q_1=\{ \xi_{p+q+1}^*,\ldots,\xi^*_{p+q+N}\}$. 
		
		Moreover, up to a subsequence  the following estimate holds:  as $\eps\to 0$
		\[  \rho_\eps \frac{ 2 \tilde{K} e^{-h_Q}e^{w_\eps}}{\int_\Sigma  2 \tilde{K} e^{-h_Q}e^{w_\eps}\, \d v_g}  \, \d v_g  \stackrel{*}{\rightharpoonup}   \sum_{i=1}^{p+q+N} (1+\gam_i )\varrho(\xi^*_i) \bm{\delta}_{\xi^*_i}.  \]
	\end{corollary}
	
	The remainder of this paper is organized as follows.
	
	In Section \ref{sec:pre}, we present the necessary preliminaries, including isothermal coordinates, Green’s functions, projected bubbles,  the Moser-Trudinger inequality and the approximate solution employed in the analysis. Section \ref{sec:3} is devoted to a finite-dimensional reduction of the problem using these approximate solutions. We first analyze the linearized problem and establish its partial invertibility, and then address the infinite-dimensional problem using a fixed-point theorem. In Section \ref{sec:4}, we expand the reduced function and analyze its properties. Section  \ref{sec:5} contains the proof of the main results using variational methods. Finally, Appendices \ref{app:a} and \ref{app:b} provide technical details on the linearized problem and on the asymptotic behavior of the projected bubbles, respectively.\\

	\noindent	{\bf Acknowledgement.} M. Ahmedou acknowledges support from the DFG grant AH 156/2-1.\\
	\section{Preliminaries}\label{sec:pre}
	\subsection*{Notation}
	Throughout this paper, we use the terms ``sequence" and ``subsequence" interchangeably, as the distinction is not crucial to the context of our analysis. The constant denoted by $C$ in our arguments may assume different values across various equations or even within different lines of equations.   
	
	We use the standard asymptotic notation: $f(x) \asymp g(x)  \text{ as } x \to a
	$ means there exists a constant $C>0$ such that   $
	\frac 1 C |g(x)|\leq |f(x)|\leq C|g(x)|
	$; $f(x) = O(g(x))  \text{ as } x \to a$
	means there exist constants $C > 0 $   such that
	$
	|f(x)|   \leq C|g(x)|
	$;
	$	f(x) = o(g(x))  \text{ as } x \to a
	$ means 
	$
	\lim_{x \to a}  |f(x)|/ |g(x)| = 0.$

	\subsection{Isothermal coordinates and Green functions}\label{sec:iso}
	Every Riemann surface $(\Sigma, g)$ is conformally flat in a local sense, with isothermal coordinates (see \cite{C1955, HW1955, V1955, L1957}) where $g$ is conformal to the Euclidean metric. This paper uses a family of isothermal coordinates  from \cite{EF2014, YZ2021}, mapping an open neighborhood of $\xi$ to a disk or half-disk in $\mathbb{R}^2$, as outlined below:
	
	For any $\xi\in\intsigma$,  there exists an isothermal coordinate system $(U(\xi), y_{\xi})$ such that $y_{\xi}$ maps an open neighborhood $U(\xi)$ around $\xi$ onto an open disk $B^{\xi}$ with radius $2 r_{\xi}$ in which $g=e^{\hat{\varphi}_{\xi}}\langle\cdot,\cdot\rangle_{\R^2}$ satisfying that  $y_{\xi}(\xi)=(0,0)$ and $\overline{U(\xi)} \subset \intsigma$.
	
	For $\xi \in \partial\Sigma$, there exists an isothermal coordinate system $(U(\xi), y_{\xi})$ around $\xi$ such that the image of $y_{\xi}$ is a half disk $B^{\xi}:=\{y=(y_1, y_2) \in  \R^2:|y|<2 r_{\xi}, y_2 \geq 0\}$ of a radius $2 r_{\xi}$ and $y_{\xi}(U(\xi) \cap \partial\Sigma)=\{y= (y_1, y_2) \in \R^2:|y|<2 r_{\xi}, y_2=0\}$ with $g=e^{\hat{\varphi}_{\xi}}\langle\cdot,\cdot\rangle_{\R^2}$ satisfying that $y_{\xi}(\xi)=(0,0)$.
	
	Moreover, $\hat{\varphi}_{\xi}: B^{\xi} \to \R$ is related to  the Gaussian curvature $K_g$ of $\Sigma$ by the following equation: 
	$$
	-\Delta \hat{\varphi}_{\xi}(y)=2 K_g(y_{\xi}^{-1}(y)) e^{\hat{\varphi}_{\xi}(y)} \quad \text { for all }\, y \in B^{\xi}.
	$$
	Moreover, if $\xi\in \partial\Sigma $, then
	$$\frac{\partial}{\partial y_2}\hat{\varphi}_{\xi}(y)=-2k_g(y)e^{\frac{\hat{\varphi}_{\xi}(y)}{2} } \quad \text { for all }\, y \in \partial B^{\xi}\cap \partial\R_+^{2},$$
	where $k_g$ is the geodesic curvature of the boundary $\partial\Sigma$ with respect to the metric $g$. 
	Both $y_{\xi}$ and $\hat{\varphi}_{\xi}$ are assumed to depend smoothly on $\xi$  as in \cite{EF2014}. In addition, $\hat{\varphi}_{\xi}$ satisfies   
	\begin{equation}\label{varphixi}
		\hat{\varphi}_{\xi}(0)=0  \text{ and } \nabla\hat{\varphi}_{\xi}(0)=\left\{\begin{aligned}
			&0& & \text{ if } \xi\in \intsigma, \\
			&(0,-2k_g(\xi)) && \text{ if } \xi\in \partial \Sigma.
		\end{aligned}\right. 
	\end{equation}
	Set $\varphi_{\xi}(x)=
		\hat{\varphi}_\xi(y_{\xi}(x)) $  for $ x\in U(\xi)$.
	The transformation law for $\Delta_g$ under conformal maps is given as follows: for
	$\tilde{g}=e^{\varphi} g$,  
	$
	\Delta_{\tilde{g}}=e^{-\varphi} \Delta_g.
	$
	Moreover, the isothermal coordinates preserve the Neumann boundary conditions, i.e.  for any $x \in U(\xi)\cap \partial\Sigma$, we have
	$$
	(y_{\xi})_*(\nu_g(x))= -e^{ -\frac{\hat{\varphi}_{\xi}(y)}2}\frac {\partial} { \partial y_2 }\Big|_{	y=y_{\xi}(x)}.
	$$
	For $\xi \in \Sigma$ and $0<r \leq 2 r_{\xi}$, we set
	$
	B_r^{\xi}:=B^{\xi} \cap\{y \in \R^2:|y|<r\} $ and $U_r(\xi):=y_{\xi}^{-1}(B_r^{\xi}).
	$
	
	For any $\xi\in \Sigma $, we define the Green function associated with the Neumann Laplacian  $G^g(\cdot,\xi):=G^g_\xi$ for \eqref{eq:mf} by following equations:
	\begin{equation}\label{Green}
		\left\{\begin{aligned}
			-\Delta_g G^g(x,\xi)  &=\bm{\delta}_\xi -1, & & x\in \Sigma,\\
			\partial_{ \nu_g } G^g(x,\xi) &=0, & &x\in \partial \Sigma, \\
			\int_{\Sigma } G^g(x,\xi)  \, \d v_g(x)& =0.&&
		\end{aligned}
		\right.
	\end{equation}
	Let $\chi$ be a radial cut-off function in  $C^{\infty}(\R,[0,1])$ such that
	\begin{equation}
		\label{eq:cut_off} \chi(s)=\begin{cases}
			1 &\text{ if }\,|s|\leq 1,\\
			0 & \text{ if }\, |s|\geq 2 .
		\end{cases}
	\end{equation}	
	Set $\chi_{\xi}(x):=\chi(4|y_{\xi}(x)|/r_{\xi})$.  Define 
	$
	\Gamma^g(x,\xi)= \Gamma^g_{\xi}(x)= -\frac{4}{\varrho( \xi)} \chi_{\xi}(x) \ln  |y_{\xi}(x)|.
	$
	The regular part of the Green function is defined by 
	$
	H^g(x,\xi)=H^g_{\xi}(x):= G^g(x,\xi)-\Gamma^g(x,\xi), 
	$
	which solves the following equations: 
	\begin{equation}\label{eq:H_xi}
		\left\{\begin{aligned}
			-\Delta_g H^g_\xi & =- \frac{4}{\varrho( \xi)}(\Delta_g \chi_{\xi}) \ln   |y_{\xi}|-\frac{8}{\varrho( \xi)}\langle\nabla \chi_{\xi}, \nabla \ln   |y_{\xi}|\rangle_g-1 &&\text { in }\, \Sigma,\\ 
			\partial_{\nu_g} H^g_\xi&=\frac{4}{\varrho( \xi)}(\partial_{\nu_g} \chi_{\xi}) \ln   |y_{\xi}|+\frac{4}{\varrho( \xi)} \chi_{\xi} \partial_{\nu_g} \ln   |y_{\xi}| &&\text { on }\, \partial \Sigma  , \\
			\int_{\Sigma } H^g_\xi \, \d v_g  &=\frac{4}{\varrho( \xi)} \int_{\Sigma } \chi_{\xi} \ln   |y_{\xi}| \, \d v_g.& &
		\end{aligned}\right.
	\end{equation}
	By the regularity of the elliptic equations (refer to \cite{N2014,ABH2024MF,A1959,W2004}, for instance), there is a unique  solution $H^g(x, \xi)$ of  \eqref{eq:H_xi} in $C^{2, \al}(\Sigma)$ for any $\alpha\in (0,1)$. $H^g(x, \xi)$ is the regular part of $G^g(x, \xi)$ and $R^g(\xi):=H^g(\xi, \xi)$ is the Robin function on $\Sigma$. It is evident that   $H^g(\xi, \xi)$  does not depend on the cut-off function $\chi$ and the local chart.
	
	For $p,q\in \mathbb{N}$,   
	set $\Delta_{p,q}(\Sigma):=\{\xi=(\xi_1,\ldots,\xi_{p+q})\in \intsigma^p \times (\partial\Sigma)^{q}: \xi_i=\xi_j \text{ for some } i\neq j\}$, 
	$ \Xi_{p,q}:=\intsigma^p \times (\partial\Sigma )^{q}\setminus \Delta_{p,q}(\Sigma),
	$
	and 
	$
	\Xi^\prime_{p,q}:=
	\{ \xi=(\xi_1,\ldots,\xi_{p+q}):   \xi_i \notin Q  \text{ for all } i=1,\ldots, p+q\}. 
	$
	Denote $\R_+=(0,+\infty)$. Given $m\in \N$ and $(p,q,Q_1)\in \cI_m$, denote  $Q_1=\{\xi^*_{p+q+1}, \ldots, \xi^*_{p+q+N}\}$. 
	For $i=1,\ldots,  p+q+N$, there exists a  smooth positive function  $K_i\in C^\infty(\Sigma, \R_+)$ such that  in a small neighborhood of $\xi_i:=\xi^*_i$,
	\begin{equation}
		\label{def:Ki} K(x)= K_i(x)|y_{\xi_i}(x)|^{2\gam_i}. 
	\end{equation}
	In particular, for $i=1,\ldots, p+q$,  we take $K_i(x)=K(x)$. 
	Define  the reduced function on the configuration set  $\Xi^\prime_{p,q}$ as follows: 
	\begin{align}  \label{def:reduce_fun_0}
		\ff(\xi^0)&=\sum_{i=1}^{p+q+N} 2(1+\gam_i)\varrho( \xi_i)\ln   K_i(\xi_i) +\sum_{i=1}^{p+q+N}(1+\gam_i)^2  \varrho( \xi_i)^2 R^g(\xi_i) \\
		&+ \sum_{\mytop{
				i,j=1,\ldots,p+q+N}{
				i\neq j} }(1+\gam_i)(1+\gam_j)\varrho( \xi_i)\varrho( \xi_j) G^g(\xi_i,\xi_j), 
		\nonumber
	\end{align}
	where $R^g$ is  the Robin function defined in Section~\ref{sec:iso}, $\xi^0:=(\xi_1,\ldots,\xi_{p+q})\in \Xi^{\prime}_{p,q}, N:=|Q_1|$, $Q_1:=\{ \xi^*_{p+q+1}, \ldots, \xi^*_{p+q+N}\}$, and  $\xi:=(\xi_1,\ldots,\xi_{p+q+N})=(\xi^0,Q_1)$. 
	
	
	\subsection{Projected bubbles}
	The problem \eqref{eq:s_mf} has the following  local limit profile: with   scaling center $\xi$ in $\intsigma$
	\begin{equation*}\label{eq:Liou_1}
		\left\{	\begin{aligned}
			-\Delta \tilde{\delta}_{\lam}^\gam&=|y|^{2 \gam} e^{\tilde{\delta}_{\lam}^\gam} &&\text { in } \mathbb{R}^2, \\
			\int_{\R^2}|y|^{2 \gam} e^{\tilde{\delta}_{\lam}^\gam} \d y&=8 \pi(1+\gam),  &&
		\end{aligned} \right. 
	\end{equation*}
	with scaling center $\xi$ on $\partial\Sigma$, 
	\begin{equation}\label{eq:Liou_2}
		\left\{ \begin{aligned}
			-\Delta \tilde{\delta}_{\lam}^\gam&=|y|^{2 \gam} e^{\tilde{\delta}_{\lam}^\gam} &&\text { in } \R_+^2, \\
			\partial_{y_2} \tilde{\delta}_{\lam}^\gam&=0& &\text{ on }\partial \R_+^2 ,  \\
			\int_{\R^2_+}|y|^{2 \gam} e^{\tilde{\delta}_{\lam}^\gam} \d y&=4 \pi(1+\gam), &&
		\end{aligned} \right. 
	\end{equation}
	for  $\gam>-1$.
	All solutions of the Liouville-type equations above are completely classified 
	by Chen and Li \cite{CL1991} for the case $\gam = 0$, and by Prajapat and 
	Tarantello \cite{PT2001} for $\gam \neq 0$. 
	
	For  $\gam\in (-1,\infty) \setminus \N,$
	any solution $\tilde{\delta}_{\lam}^{\gam}$ of \eqref{eq:Liou_1} or 
	\eqref{eq:Liou_2} can be written in the following form:
	\[ \tilde{\delta}^\gam_{\lam}(y):= \tilde{\delta}^\gam(\lam y)+\ln (\lam^{2(1+\gam)}) =\ln  \frac{ 8(1+\gam)^2 \lam^{2(1+\gam)}}{( 1+ \lam^{2(1+\gam)}|y|^{2(1+\gam)})^2} \quad \text{ for } y\in \R_{\xi}, \]
	where 
	$ \tilde{\delta}^\gam(y):=\ln  \frac{ 8(1+\gam)^2 }{( 1+ |y|^{2(1+\gam)})^2},
	$
	$\lam>0,$ and $ \R_{\xi}=\R^2$ when $\xi\in\intsigma$; $\R_{\xi}= \R_+^2$ when $\xi\in \partial\Sigma.$ 
	Denote that 
	\[\delta^\gam_{\lam,\xi}(x):=\tilde{\delta}^\gam_{\lam}( y_{\xi}(x))=\ln \frac{8(1+\gam)^2\lam^{2(1+\gam)}}{(1+\lam^{2(1+\gam)}|y_{\xi}(x)|^{2(1+\gam)})^2}, \]
	for all $x\in U(\xi)$, and $\delta^\gam_{\lam,\xi}(x)=0$ for all $x\in \Sigma\setminus U(\xi).$
	Then, we define the projected bubbles by the following equations: 
	\begin{equation}\label{eq:proj_bubble}
		\left\{\begin{aligned}
			-\Delta_g P  \delta^\gam_{\lam,\xi}&=   \chi_{\xi} e^{-\varphi_{\xi}}|y_{\xi}|^{2\gam}e^{\delta^\gam_{\lam,\xi}}-\overline{  \chi_{\xi} e^{-\varphi_{\xi}} |y_{\xi}|^{2\gam}   e^{\delta^\gam_{\lam,\xi}}  } &&  \text{ in }\, \Sigma,\\
			\partial_{ \nu_g } P  \delta^\gam_{\lam,\xi}&=0 && \text{ on }\, \partial \Sigma , \\
			\int_{\Sigma } P  \delta^\gam_{\lam,\xi}  \, \d v_g&=0, &&
		\end{aligned}\right.
	\end{equation}
	where $\overline{f}:=\int_{\Sigma} f \, \d v_g$ for any $f\in L^1(\Sigma).$ Using the regularity theory in \cite{W2004, A1959}, $P\delta^\gam_{\lam,\xi}$ is well-defined and smooth.

	\subsection{The Moser-Trudinger inequality}
	The Moser-Trudinger inequality is an important tool for studying Liouville-type equations.
	
In the regular case, the classical Moser-Trudinger inequality holds; see, for instance,
\cite{yang2006extremal,MalchiodiRuiz2011,Moser1971,ChangYang1988}. We recall it in the following lemma.
	\begin{lemma} \label{lem:M-T}
		For a closed compact Riemann surface, $\Sigma$, there exist positive constants $C$ such that 
		\[ \ln \int_\Sigma K e^{u} \, \d v_g\leq \frac{1 }{16\pi } \int_\Sigma  |\nabla u|^2\, \d v_g + 
		C, \]
		for  all 
		$u \in \oH$;
		for a compact Riemann surface $\Sigma$ with boundary 
		there exist positive constants $C$ such that 
		\[ \ln \int_\Sigma K e^{u} \, \d v_g\leq \frac{1 }{8\pi } \int_\Sigma  |\nabla u|^2\, \d v_g + 
		C, \]
		for  all 
		$u \in \oH$;
	\end{lemma}
	
	In the singular case with negative singular sources, the sharp Moser-Trudinger
	constant is not yet fully understood.
	Nevertheless, for the purposes of the present paper, it suffices to use a
	non-optimal version of the inequality, stated below.
	\begin{lemma} 	\label{lem:M-T2}
		There exist positive constants $C,\tau^*>0$ such that 
		\[ \ln \int_\Sigma K e^{u} \, \d v_g\leq \frac{1 }{4\mathrm{b}} \int_\Sigma  |\nabla u|^2\, \d v_g + 
		C, \]
		for  all 
		$u \in \oH$ and $\mathrm{b}\in (0, 2\pi \tau^*)$,
	\end{lemma}
	If $\gam(\xi) \ge 0$ for all $\xi \in  Q$, Lemma~\ref{lem:M-T} follows from \cite{yang2006extremal}.  
	In the presence of negative conical singularities, the inequality follows by combining the
	argument of \cite{Troyanov1991} with a doubling construction.  
	We omit the details.

	\subsection{Approximate solutions}
	Suppose that $\{u_n\}$ is a sequence of solutions of the following homogeneous Neumann boundary problem: 
	\[   	\left\{ \begin{aligned}
		-\Delta_g u_n &=\rho_n  \Big( \frac{K^n e^{u_n}}{\int_{\Sigma} K^n e^{u_n}\, \d v_g} -1\Big)  && \text{ in } \Sigma, \\
		\partial_{\nu_g} u_n&=0 & &\text{ on } \partial\Sigma, 
	\end{aligned}\right.  \]
	with $K^n \to K$ and $\rho_n\to\rho.$
	We have the following concentration-compactness alternative for singular mean field equations.
	\begin{lemma}
		\label{thm:C-C}
		One of the following alternatives occurs:
		\begin{itemize}
			\item [a.] (Compactness) $\{u_n\}$ is uniformly bounded in $H^1(\Sigma)$. 
			\item [b.](Concentration) The blow-up point set $\cS:=\{ x\in \Sigma: \exists x_n \to x \text{ s.t. } u_n(x_n)\to+\infty\}$ is non-empty and finite, and $u_n -\ln  \int_\Sigma K^n e^{u_n} \, \d v_g\to -\infty $ is locally uniform in $\Sigma\setminus \cS.$ Moreover, 
			\[  \rho_n  \frac{K^n e^{u_n}}{\int_\Sigma K^n e^{u_n} \, \d v_g} \, \d v_g \stackrel{*}{\rightharpoonup}  \sum_{x\in \cS} (1+\gam(x)) \bm{\delta}_x, \]
			where $\bm{\delta}_x$ is the Dirac mass concentrated at point $x$ and 	$
			\gam(x) = 
			\begin{cases} 
				0 & \text{ if } x \notin Q\\
				\gam(x)& \text{ if } x\in Q
			\end{cases}. $
		\end{itemize} 
	\end{lemma}
	
	The concentration-compactness alternative for Liouville-type equations is a 
	well-studied result, initially formulated by Brezis and Merle 
	\cite{brezis_uniform_1991}. Subsequent developments have extended this 
	framework to various settings.
	In the case of regular problems, we refer to 
	\cite{Lucia2002Vortex, suzuki1992emden} for the case without boundary 
	blow-up, and to \cite{LSY2023, WW2002steady, Battaglia2019stationary} for 
	those involving boundary blow-up. For singular problems, the case without 
	boundary blow-up has been treated in 
	\cite{Bartolucci2002Liouville, BartolucciTarantello2002} and the references 
	therein.
	Following the argument in \cite{LSY2023}, and with suitable modifications to 
	account for singular sources, Theorem~\ref{thm:C-C} can be established by
	an analogous approach.

	Using Lemma~\ref{thm:C-C} together with the Moser-Trudinger inequality in Lemma \ref{lem:M-T}, we obtain that  
	\begin{itemize}[noitemsep, topsep=0pt, leftmargin=*]
		\item[] \emph{ $\cS \neq \emptyset$ if and only if 
			$
			\displaystyle \int_{\Sigma} K^{n} e^{u_{n}}\, \d v_{g} \to +\infty.
			$}
	\end{itemize}

	To construct a sequence of blow-up solutions to \eqref{eq:mf} as the parameter $\rho\to 4\pi n_m$, it is equivalent to finding blow-up solutions of the following Liouville-type equations:
	\begin{equation}
		\label{eq:main_L}
		\left\{\begin{aligned}
			-\Delta_g u&= \eps^2 Ke^u -\overline{\eps^2 K e^u}  && \text{ in } \Sigma\\
			\partial_{\nu_g} u&=0 & &\text{ on }\partial\Sigma
		\end{aligned}, \right. 
	\end{equation}
	by   Lemma \ref{thm:C-C}  and the change of the variable $\eps^2=\frac{\rho}{\int_\Sigma K e^u \, \d v_g} \to 0.$
	The Euler-Lagrange functional associated with \eqref{eq:main_L} is given by
	\begin{equation*}
		E_\eps(u)=\frac 1 2 \int_\Sigma |\nabla u|_g^2 \, \d v_g - \eps^2 \int_\Sigma K e^u \, \d v_g \qquad \text{ for }  u\in \oH. 
	\end{equation*}

	Next, for any fixed $m\in\N$ and $(p,q,Q_{1})\in\cI_{m}$, we  construct a family of blow-up solutions 
	$\{u_{\eps}\}$ to \eqref{eq:main_L} as $\eps\to0$.  
	We then prove that the same family $\{u_{\eps}\}$ also yields a blow-up sequence for the mean field equation \eqref{eq:mf} as the parameter $\rho \to 4\pi n_{m}$.

	Suppose that $\xi^{0,*}:=(\xi_1^*,\ldots,\xi_{p+q}^*)$ is a $C^1$-stable critical point of $\ff$.  
	For each $\xi_i^*$, let $(y_{\xi_i^*},U(\xi_i^*))$ be an isothermal coordinate chart such that  
	$U_\sigma(\xi_i^*) \subset U_{r_{\xi_i^*}}(\xi_i^*)$ for all sufficiently small $\sigma>0$.  
	Then, for any $\xi_i\in U_{\sigma}(\xi_i^*)$, there exists an isothermal coordinate system 
	$(y_{\xi_i},U(\xi_i))$ as in \cite[Claim~3.1]{ABH2024MF}.  
	In particular, both $y_{\xi_i}$ and $\hat{\varphi}_{\xi_i}$ depend smoothly on $\xi_i$ for 
	$\xi_i \in U_{\sigma}(\xi_i^*)$, $i=1,\ldots,p+q+N$.
	
	Define
	\begin{equation}\label{eq:def_M_xi_delta}
		M_{\sigma,\xi^*}
		:=\bigl\{\,\xi=(\xi_1,\ldots,\xi_{p+q},Q_1)\in\Xi_{p,q}' : 
		\xi_i\in U_\sigma(\xi_i^*) \text{ for } i=1,\ldots,p+q \,\bigr\},
	\end{equation}
	for $\sigma>0$, 
	and
	$
	M_{\sigma,\xi^*}^0
	:=\bigl\{\,\xi^0 \in \Xi_{p,q}' : (\xi^0,Q_1)\in M_{\sigma,\xi^*} \,\bigr\},
	$
	where $\xi^*=(\xi^{0,*},Q_1)$.
	
	For $(p,q,Q_1)\in\cI_m$, set $N:=|Q_1|>0$,  
	$
	Q_1=\{\xi_{p+q+1}^*,\ldots,\xi_{p+q+N}^*\},
	$ $
	\xi=(\xi_1,\ldots,\xi_{p+q+N})=(\xi^0,Q_1).
	$
	For any $\xi\in M_{\sigma,\xi^*}$, we may choose a uniform constant $r_0>0$, depending only on $\sigma$,  
	such that $r_{\xi_i}\ge 4r_0$ for all $i=1,\ldots,p+q+N$.  
	We further assume that
	\[
	U_{4r_0}(\xi_i)\cap U_{4r_0}(\xi_j)=\emptyset\quad\text{for } i\neq j,
	\qquad
	U_{2r_0}(\xi_i)\cap\partial\Sigma=\emptyset\quad\text{for } i=1,\ldots,p.
	\]
	We also denote
	$
	\ii(\xi) = 2, \text{ for } \xi\in\intsigma; \ii(\xi)=
	1, \text{ for }\xi\in\partial\Sigma.
	$

	We define the scaling functions in $\Sigma\setminus\{ \xi_j: j\neq i\}$  for $i=1,\ldots, p+q+N$ by 
	\begin{equation}
		\label{def:d_i}
		d_i(x):= \frac  1{ 8(1+\gam_i)^2} e^{(1+\gam_i) \varrho(\xi_i)H^g(x, \xi_i) +\sum^{p+q+N}_{\mytop{j=1}{j\neq i}}(1+\gam_j)\varrho(\xi_j) G^g(x,\xi_j) +\ln  K_i(x)},
	\end{equation}
	where $K_i$ is given by \eqref{def:Ki}. 
	Let 
	\begin{equation}\label{def:lam_i}
		\lam_i^{-2(1+\gam_i)}=d_i \eps^2,
	\end{equation}
	where $d_i:=d_i(\xi_i).$
	For simplicity, we use the notation   $\chi_i:=\chi_{\xi_i}$, $\varphi_{\xi_i}:=\varphi_i$, $\hat{\varphi}_{\xi_i}:=\hat{\varphi}_i$,  
	$\delta_i=
	\delta^{\gam_i}_{\lam_i,\xi_i} $
	and $P \delta_i= P\delta_{\lam_i,\xi_i}^{\gam_i}$
	for $i=1,\ldots, p+q+N$. 
	

	We consider the following approximate manifold:
	\begin{align*}
		\quad \cM^\eps_{p,q,Q_1}&=\Big\{\sum_{i=1}^{p+q+N} P\delta_i: \xi^0\in M^0_{\sigma,\xi^*}, \xi_{p+q+j}=\xi^*_{p+q+j} \text{ for } j=1,\ldots, N; \lam_i \text{ is  given by } \eqref{def:lam_i} \Big\}, 
	\end{align*}
	where $\eps>0$ is a parameter and  $Q_1=\{ \xi^*_{p+q+1},\ldots,\xi^*_{p+q+ N}\}\subset Q. $
	It is clear that $\cM^\eps_{p,q,Q_1}$ is a $(2p+q)$-dimensional manifold. The tangent space is generated by 
	\begin{equation*}
		\partial_{(\xi_i)_j} P\delta_i, \quad\forall i=1,\ldots, p+q, j=1,\ldots,\ii(\xi_i).
	\end{equation*}
	
	Given $(p,q,Q_1)\in \cI_m$ with $|Q_1|=N$, set 
	$\xi=(\xi_1,\ldots,\xi_{p+q+N})=(\xi^0,Q_1)$.  
	Define
	\[
	W := \sum_{i=1}^{p+q+N} P\delta_i,
	\]
	which serves as the approximate solution, depending on the variable 
	$\xi^0 \in \Xi_{p,q}'$.  
	We then look for a solution of the form
	$
	u = W + \phi,
	$
	where $\phi$ is the remainder term and satisfies $\|\phi\| =
(	\int_\Sigma |\nabla \phi|_g^2 \, \d v_g)^{\frac 1 2 } = o(1)$ as 
	$\eps \to 0$.

	\section{Finite-dimensional reduction}\label{sec:3}
	In this section, we employ the ansatz $W=\sum_{i=1}^{p+q+N}P\delta_i$ as an approximate solution to
	problem~\eqref{eq:main_L} in order to perform a finite-dimensional reduction.  
	We begin by establishing a key lemma concerning the partial invertibility of the linearized
	operator and by identifying an appropriate finite-dimensional subspace of $\oH$ for the
	reduction scheme.  
	Subsequently, we estimate the corresponding error term and apply a fixed-point argument to
	reduce the original problem to a finite-dimensional one.

	\subsection{Linearized problem}

	Utilizing the Moser-Trudinger  inequality from Lemma \ref{lem:M-T} on compact Riemann surfaces,  the map $$\overline{\mathrm{H}}^1\rightarrow L^s(\Sigma ),\quad  u\mapsto e^u$$ is continuous.  	For any $s>1$, let $\tilde{i}^*_s: L^s(\Sigma)\rightarrow \overline{\mathrm{H}}^1$  be the adjoint operator corresponding to the immersion $i: \overline{\mathrm{H}}^1 \rightarrow L^{\frac{s}{s-1}}(\Sigma)$ and $\tilde{i}^* : \cup_{s>1} L^s(\Sigma)\rightarrow \overline{\mathrm{H}}^1$. For any $f\in L^s(\Sigma)$, we define that $i^*(f):=\tilde{i}^*( f-\bar{f})$, i.e.  for any $h\in \overline{\mathrm{H}}^1 $, 	$$\langle i^*(f), h\rangle =\int_{\Sigma}K (f-\bar{f}) h \, \d v_g.$$

	By substituting  $W+\phi$ into \eqref{eq:main_L}, we have 
	\begin{equation}
		\label{eq:main_L2} \left\{   \begin{aligned}
			\cL(\phi)&=S(\phi)+ N(\phi)+ \cR,&& \text{ in } \Sigma,\\
			\partial_{\nu_g}\phi &=0, & &\text{ on } \partial\Sigma, 
		\end{aligned}\right. 
	\end{equation}
	where $$\cL(\phi):= -\Delta_g \phi -\sum_{i=1}^{p+q+N} \chi_i  e^{-\varphi_i}  e^{\delta_i} \phi  + \overline{\sum_{i=1}^{p+q+N} \chi_i  e^{-\varphi_i}  e^{\delta_i} \phi} $$ is the first linear term, 
	\begin{equation}
		\label{def:S} S(\phi):= \eps^2 K e^{W}\phi-\sum_{i=1}^{p+q+N} \chi_i  e^{-\varphi_i}  e^{\delta_i} \phi  + \overline{ \eps^2 K e^{W}\phi-\sum_{i=1}^{p+q+N} \chi_i  e^{-\varphi_i}  e^{\delta_i} \phi }
	\end{equation}
	is the higher linear term, 
	\begin{equation}
		\label{def:nonlinear} N(\phi):= \eps^2 e^W( e^\phi-\phi-1) -\overline{ \eps^2 e^W( e^\phi-\phi-1) }
	\end{equation}
	and 
	\begin{equation}
		\label{def:R}   \cR:= \Delta_g W+ \eps^2 K e^W -\overline{\eps^2 K e^W }
	\end{equation}
	is the remainder term. 
	To find a sequence of blow-up solutions of \eqref{eq:main_L} with the form $W+\phi$, it is sufficient to find a proper $\xi\in \Xi_{p,q}^\prime\times Q_1^N$ and $\phi$ solving the nonlinear problem \eqref{eq:main_L2} for $\eps>0$ sufficiently small.  
	
	From the construction, we know
	\[ \int_\Sigma \cL(\phi)\, \d v_g =\int_\Sigma S(\phi) \, \d v_g=  \int_\Sigma N(\phi) \, \d v_g = \int_\Sigma \cR \, \d v_g. \]
	Next, we proceed to show the partial invertibility of the linearized operator on a subspace of $\oH$ with finite co-dimension, which is the key lemma for the  Lyapunov-Schmidt reduction process. 
	
	Let 	$$z^0_0(y)=2\frac{ 1-|y|^2}{1+|y|^2} \quad \text{ and }\quad z^0_{j}(y)=\frac{4y_j}{1+|y|^2} \text{ for }j=1,\ldots,\ii(\xi),$$
	which  generate the kernel of the following linear problem (see \cite{CL2002,Esposito2005}):
	\begin{equation*}
		\left\{   \begin{aligned}
			-\Delta \phi &= \frac{8}{(1+|y|^2)^2}   \phi  & &\text{ in } \R_\xi\\
			\partial_{y_2}\phi &=0 &  &\text{ on } \partial \R_\xi
		\end{aligned}\right. ,
	\end{equation*}
	where $\R_\xi=\R^2$ for $\xi\in\intsigma$ and $\R^2_+$ for $\xi\in \partial \Sigma$.  Furthermore, for $\gam\in (-1, +\infty)\setminus (\N\cup \{ 0\})$,   we set 
	$$z_0^{\gam}= 2(1+\gam) \frac{1-|y|^{2(1+\gam)}}{1+|y|^{2(1+\gam)}},$$ 
	which generates the kernel of the following linear problem (see Lemma \ref{lem:profile_lin}):
	\begin{equation*}
		\left\{   \begin{aligned}
			-\Delta \phi &=   |y|^{2(1+\gam )}\frac{8}{(1+|y|^{2(1+\gam)})^2}   \phi  && \text{ in } \R_\xi\\
			\partial_{y_2}\phi &=0 &&  \text{ on } \partial \R_\xi
		\end{aligned}\right. . 
	\end{equation*}
	Then, we pull back the function $z^{\gam_i}_j$ to the Riemann surface by isothermal coordinates, 
	\[ Z^j_i(x)= \left\{   \begin{aligned}
		&	z^{\gam_i}_j (\lam_i |y_{\xi_i}|)  && x\in U(\xi_i)\\
		&	0&  &x\in \Sigma\setminus U(\xi_i)
	\end{aligned}\right. , \]
	for $i=1,\ldots, p+q+N$  and $j= 0,\ldots,\ii(\xi_i)$ for $i=1,\ldots, p+q$ and $j=0$ for $i=p+q+1, \ldots, p+q+N$. 
	Then, we projected the function $Z^j_i$ into the space $\oH$ by following equations 
	\begin{equation}\label{def:PZ}
		\left\{  \begin{aligned}
			-\Delta_g PZ^j_i&= \chi_i  |y_{\xi_i}|^{2\gam_i} e^{-\varphi_i} e^{\delta_i}  Z^j_i -\overline{\chi_i  |y_{\xi_i}|^{2\gam_i} e^{-\varphi_i}   e^{\delta_i}  Z^j_i } & &\text{ in } \Sigma  \\
			\partial_{\nu_g} PZ^j_i &=0 & &\text{ on } \partial \Sigma
		\end{aligned}.   \right. 
	\end{equation}
	Then, we will consider the linearized problem on 
	the orthogonal space of 
	\[ K_{\xi}:= \la  PZ^j_i: i=1,\ldots, p+q, j=1, \ldots, \ii(\xi_i)  \ra, \]
	and its orthogonal space
	\[  K_{\xi}^{\perp}:=\left\{ \phi \in\oH: \la \phi, h\ra =0, \forall h \in K_{\xi}\right\}. \]
	For any $s>1$, given $h\in L^{s}(\Sigma):=\{h: \|h\|_{s}:=(\int_\Sigma |h|^s \, \d v_g )^{\frac 1 s }<+\infty\}$, find $\phi\in K_{\xi}^\perp\cap W^{2,s}(\Sigma)$  such that 
	\begin{equation}~\label{eq:linprob}
		\left\{   \begin{aligned}
			\cL(\phi)&= h & &\text{ in }\Sigma\\
			\partial_{ \nu_g }\phi&= 0&& \text{ on }\partial \Sigma\\
		\end{aligned}\right. . 
	\end{equation}
	
We first obtain the key lemma for the Lyapunov-Schmidt reduction process. It is rather standard to deduce it following the argument in \cite{EF2014,del_pino_singular_2005}. For brevity, we provide a brief outline of the proof.   
	\begin{lemma}
		\label{lem:invertible}
		Let $\Dc$ be a compact subset of $M_{\sigma, \xi^*}$. 
		For any $s>1$,  
		there exist $\eps_0 > 0$ and $C > 0$ such that for any $\eps \in (0, \eps_0)$, $\xi\in \Dc$,  $h\in L^s(\Sigma)$ and   $\phi \in \oH\cap K_{\xi}^{\perp}$ is the unique solution of \eqref{eq:linprob}, the following estimate holds true: 
		\[
		\|\phi\| \leq C |\ln  \eps|\,  \|h\|_s,
		\]
		where $\|h\|_s:=\|h\|_{L^s(\Sigma)}$.
	\end{lemma}
	
	\begin{proof}
		We will prove it by contradiction. Suppose Lemma~\ref{lem:invertible} fails, i.e.  there exist $s>1$, a sequence of $\eps_n\rightarrow 0, \xi^n\rightarrow \xi^*$, $h_n\in L^s(\Sigma)$ and $\phi_n\in  \oH \cap K_{\xi}^{\perp}$ solves \eqref{eq:linprob} for $h_n$ satisfying that 	as $n\rightarrow +\infty,$
		\begin{equation}
			\label{eq:asumme_linear} \|\phi_n\|=1\text{ and } |\ln  \eps_n|\|h_n\|_s\to 0. 
		\end{equation}  
		To simplify the notation, we still use the notations $\phi, h,\xi,\eps$ instead of $\phi_{n},h_{n},\xi^n, \eps_n$. 
		
		Define that for $ i=1,\ldots,p+q+N$ , 
		\[ \tilde{\phi}_{i}(y)= \begin{cases}
			\chi\left(\frac{|y|}{ \lam_i r_0}\right)\phi\circ y_{\xi_i}^{-1}(\lam_i y),&  y \in \Omega_i:= \lam_i  B^{\xi_i}_{2r_0}\\
			0&  y\in\R_{\xi_i}\setminus\Omega_i
		\end{cases}. \]
		Then we consider the following weighted spaces for  $i=1,\ldots, p+q+N$
		\[ \rL_{i}:=\Big\{ u: \Big\|  \frac{ 2\sqrt{2}(1+\gam_i) |\cdot|^{\gam_i}}{1+|\cdot|^{2(1+\gam_i)}}  u  \Big\|_{L^2(\R_{\xi_i})} <+\infty\Big\}\]
		and 
		\[ \rH_{i}:=\Big \{u: \|\nabla u\|_{L^2(\R_{\xi_i})}+\Big\| \frac{ 2\sqrt{2}(1+\gam_i) |\cdot|^{\gam_i}}{1+|\cdot|^{2(1+\gam_i)}}  u \Big\|_{L^2(\R_{\xi_i})}<+\infty \Big\}.\]
		\begin{itemize}
			\item[Step 1.]\label{item:step1} {\it For any $i=1,\ldots,p+q+N,$  $\tilde{\phi}_{i}\rightarrow a_{i} \frac{1-|y|^{2(1+\gam_i)}}{1+|y|^{2(1+\gam_i)}} $ for some $a_i \in\R$,  
				weakly in $\rH_{i}$ and strongly in $\rL_{i}$.}
		\end{itemize}
		Applying $\phi$ as a test function for $\cL(\phi)= h$, in view of  $\int_{\Sigma}\phi \, \d v_g=0,$ we deduce that 
		\begin{equation}\label{eq:diff_h}
			\int_{\Sigma} h  \phi  \, \d v_g 	=\la \phi,\phi \ra -\sum_{i=1}^{p+q+N} \int_{\Sigma} \chi_ie^{-\varphi_i} |y_{\xi_i}|^{2\gam_i}  e^{\delta_i} \phi^2 \, \d v_g. 
		\end{equation}
		Considering that 
		$ |\int_\Sigma h\phi \, \d v_g|\leq \|h\|_s\|\phi\|=o(|\ln \eps|^{-1})$ and $\|\phi\|=1,$ we have 
		$\sum_{i=1}^{p+q+N} \| \tilde{\phi}_i\|^2_{\rL_{i}} =\cO(1),$
		as $\eps\to 0.$
		Using the assumption $\|\phi\|=1$ again,   $\|\tilde{\phi}_i\|_{\rH_{i}}^2 \leq \|\phi\|^2\leq 1$. It follows that
		$\tilde{\phi}_i\rightharpoonup \tilde{\phi}^0_i$ in $\rH_{i}$ for some $\tilde{\phi}_i^0\in \rH_{i}$, up to a subsequence. Then, the compact embedding $ \rH_{i} \hookrightarrow  \rL_{i} $ from \cite[Proposition 6.1]{GP2013} 
		implies that $ \tilde{\phi}_i\to \tilde{\phi}_i^0$ strongly in $\rL_{i}.$

		For any $\varphi\in C_c^{\infty}(\R_{\xi_i})$, assume that  $\text{supp }\varphi\subset \B_{R_0}.$
		If $\frac 1 {\lam_i} <\frac {r_0}{R_0}$, then 
		$\text{supp }\nabla \chi(\frac{|y|}{r_0}) \cap \text{supp } \varphi \left(\lam_i  y\right) =\emptyset. $
		By direct calculation, we have 
		\begin{align*}
			\int_{\R_{\xi_i}} \nabla\tilde{\phi}_{i} \nabla \varphi  \, \d y&=
			\int_{B_{2r_0}^{\xi_i}} \nabla( \chi({|y|}/{r_0}) \phi\circ y_{\xi_i}^{-1}(y))\cdot\nabla \varphi\left(\lam_i y\right) \, \d y= \int_{\Sigma}  \la \nabla \phi, \nabla ( \chi_i \varphi(\lam_i y_{\xi_i})) \ra_g \, \d v_g \\
			&=  \int_{\Sigma} \chi_i e^{-\varphi_i} |y_{\xi_i}|^{2\gam_i}e^{\delta_i} \phi\varphi \left(\lam_iy_{\xi_i} \right)\, \d v_g+ \int_{\Sigma}  \chi_i h  \varphi\left(\lam_i y_{\xi_i}\right) \, \d v_g +o(1)\\
			&=\int_{\R_{\xi_i}} \frac { 8(1+\gam_i)^2|y|^{2\gam_i} }{(1+|y|^{2(1+\gam_i)})^2}\tilde{\phi}_{i}(y)\varphi(y) \, \d y+o(1). 
		\end{align*}
		Then, we deduce that  
		$\tilde{\phi}_{i}$ converges to $\phi^0_{i}$, which is a solution of the following linear problem:
		\begin{equation}
			\label{eq:limit_linear}
			\left\{   \begin{aligned}
				-\Delta \phi&= \frac{8(1+\gam_i)^2 |y|^{2\gam_i }}{\left(1+|y|^{2(1+\gam_i)}\right)^2} \phi & & \text { in }\R_{\xi_i}, \\
				\partial_{y_2} \phi&=0 &&\text{ on } \partial\R_{\xi_i}, \\
				\int_{\R_{\xi_i}}&|\nabla \phi(y)|^2 d y<+\infty, &&
			\end{aligned}\right.
		\end{equation}
		in the distributional sense. 
		Using the regularity theory,
		$\tilde{\phi}^0_{i}$ is a $C^2$-smooth solution of \eqref{eq:limit_linear}.
		By Lemma \ref{lem:profile_lin} and \cite[Lemma D.1]{Esposito2005}, we have 
		\[ \tilde{\phi}^0_i(y)= \begin{cases}
			\sum_{j=1}^{\ii(\xi_i)} b_{ij}\frac{4 y_j}{1+|y|^2}+ a_{i} \frac{1-|y|^2}{1+|y|^2} & \text{ for } i=1,\ldots, p+q, \\
			a_{i}\frac{1-|y|^{2(1+\gam_i)}}{1+|y|^{2(1+\gam_i)}} & \text{ for } i=p+q+1,\ldots, p+q+N, 
		\end{cases}\]
		where $a_{i}$ and $b_{ij}\in \R$ are coefficients.
		Then, using  Remark \ref{rk:B.1} and the assumption $\la \phi, PZ^j_i\ra=0$ , we can deduce that $b_{ij}=0$ for $i=1,\ldots, p+q,$ and $j=1,\ldots, \ii(\xi_i)$.
		\begin{itemize}
			\item [Step 2.]\label{item:step2} {\it For $i=1,\ldots
				,p+q+N$, the following estimate holds:  as $\eps\rightarrow 0$
				\begin{equation}
					\label{eq:st2-1} \int_{\Omega_i} \frac{ 8(1+\gam_i)^2|y|^{2\gam_i}}{(1+|y|^{2(1+\gam_i)})^2}\tilde{\phi}_{i}(y) \, \d y =o(|\ln \eps|^{-1}). 
				\end{equation}
			} 
		\end{itemize}
		Applying that $PZ^0_{i}$ as a test function of \eqref{eq:linprob}, we have
		\begin{equation*}
			\int_\Sigma \chi_i |y_{\xi_i}|^{2\gam_i} e^{-\varphi_i}e^{\delta_i} Z^0_i \phi \, \d v_g= \la \phi, PZ^0_i\ra = \sum_j \int_{\Sigma} \chi_j e^{-\varphi_i}|y_{\xi_i}|^{2\gam_j} e^{\delta_j} \phi  PZ^0_i \, \d v_g +\int_\Sigma h PZ^0_j \, \d v_g. 
		\end{equation*}
		Remark \ref{rk:B.1} yields  that 
		$$
		\int_{\Sigma}  \chi_i |y_{\xi_i}|^{2\gam_i} e^{-\varphi_i} e^{\delta_i}  \phi \left( PZ^0_i-Z^0_i\right)  \, \d v_g  =-\int_{\Sigma}h PZ^0_i\, \d v_g + \cO(\eps^{\frac 2 {1+\max\{ 0,  \gam_* \}}} |\ln \eps|)= o(|\ln \eps|^{-1}).  
		$$
		Applying Remark~\ref{rk:B.1}   we derive that 
		\begin{align*}
			&	\quad(\ln  \eps)\int_{\Sigma} \chi_i |y_{\xi_i}|^{2\gam_i}  e^{-\varphi_i}e^{\delta_i}  \phi( PZ^0_i-Z^0_i)  \, \d v_g\\
			&= (\ln  \eps)
			\int_{\Sigma} \chi_i |y_{\xi_i}|^{2\gam_i} e^{-\varphi_i}e^{\delta_i}  \phi(1 + \cO( \eps^{\frac 2 {1+\max\{ 0,  \gam_* \}}} |\ln \eps| ) )  \, \d v_g\\
			&=   (\ln \eps)\int_{\Omega_i}  \frac{16(1+\gam_i)^2 |y|^{2\gam_i}}{\left(1+|y|^{2(1+\gam_i)}\right)^2} \tilde{\phi}_{i}(y)  \, \d y+o(1), 
		\end{align*}
		as $\eps\rightarrow 0.$
		Thus, \eqref{eq:st2-1} follows.
		\begin{itemize}
			\item[Step 3.]\label{item:step3} {\it  Construct the contradiction.}
		\end{itemize}	
		Using $P\delta_i$ as a test function for~\eqref{eq:linprob}, we derive that 
		\begin{equation}\label{eq:test_PU_i}
			\int_{\Sigma} \chi_i |y_{\xi_i}|^{2\gam_i} e^{-\varphi_i}e^{\delta_i} \phi \, \d v_g=\sum_j \int_{\Sigma}  \chi_j |y_{\xi_j}|^{2\gam_j} e^{-\varphi_i} e^{\delta_j}  \phi P\delta_i \, \d v_g +\int_{\Sigma} hP\delta_i \, \d v_g. 
		\end{equation}
		The L.H.S. of~\eqref{eq:test_PU_i} equals  
		$ \int_{\Omega_i}\frac{8(1+\gam_i)^2|y|^{2\gam_i}}{(1+|y|^{2(1+
				\gam_i)})^2}\tilde{\phi}_{i}(y) \, \d y
		=o(|\ln  \eps|^{-1})$, by \hyperref[item:step2]{Step 2}. 
		Lemma~\ref{lem:innner_proj} yields that  $\|P\delta_i\|= \mathcal{O}(|\ln \eps|^{\frac 1 2})$ and 
		$
		\left|\int_{\Sigma} hP\delta_i \, \d v_g \right| 
		=o(1).
		$
		Using Lemma \ref{lem:proj_bubble} and \hyperref[item:step2]{Step 2}, we have 
		\begin{align*}
			\sum_{\mytop{j=1,\ldots, p+q+N}{j\neq i}}\int_{\Sigma}  \chi_j |y_{\xi_j}|^{2\gam_j} e^{-\varphi_j}e^{\delta_j}  \phi P\delta_i \, \d v_g  &= o(1).
		\end{align*}
		Applying Lemma~\ref{lem:proj_bubble}, \hyperref[item:step1]{Step 1} and \hyperref[item:step2]{Step 2},  we deduce that as $\eps\rightarrow 0$
		\begin{align*}
			&\quad \int_{\Sigma} \chi_i |y_{\xi_i}|^{2\gam_i} e^{-\varphi_i}e^{\delta_i} \phi P\delta_i \, \d v_g\\ &=\int_{\Sigma} \chi_i |y_{\xi_i}|^{2\gam_i} e^{-\varphi_i} e^{\delta_i} \phi \Big( \chi_i \cdot(\delta^{\gam_i}_{\lam_i,\xi_i}-\ln  (8 (1+\gam_i)^2 \lam_i^{- 2(1+\gam_i)}))+(1+\gam_i) \varrho(\xi_i)H^g(\cdot, \xi_i)\\
			&\quad+\cO(\eps^{\frac{2}{1+\max\{ 0,\gam_*\}  } } |\ln \eps| )\Big)  \, \d v_g \\
			&=
			\int_{\Omega_i} \frac {8(1+\gam_i)^2 |y|^{2\gam_i}}{\left(1+|y|^{2(1+\gam_i)}\right)^2} \tilde{\phi}_{i}(y)\left(4(1+\gam_i) \ln  \lam_i -2 \ln  (1+|y|^{2(1+\gam_i)})+(1+\gam_i)\varrho(\xi_i)R^g(\xi_i)\right) d y \\
			&\quad+o(1)\\
			&\rightarrow  -2 a_{i} \int_{\R_{\xi_i}} \frac {8(1+\gam_i)^2|y|^{2\gam_i}}{\left(1+|y|^{2(1+\gam_i})\right)^2}\frac{1- |y|^{2(1+\gam_i)}}{1+|y|^{2(1+\gam_i)}}\ln (1+|y|^{2(1+\gam_i)}) \, \d y= \varrho(\xi_i)a_{i}, 
		\end{align*}
		in which the last equality used the fact that  
		$$\int_{\R^2} \frac {8(1+\gam_i)^2 |y|^{2\gam_i}}{\left(1+|y|^{2(1+\gam_i)}\right)^2}\frac{1- |y|^{2(1+\gam_i)}}{1+|y|^{2(1+\gam_i)}}\ln (1+|y|^{2(1+\gam_i)}) \, \d y=-4\pi.$$
		Therefore, $a_i=0$ for all $i=1,\dots,p+q+N$.
		
		Applying  $\phi$ as a test function for \eqref{eq:linprob}, we have 
		\begin{align*}
			1&=  \sum_{i=1}^{p+q+N} \int_\Sigma  \chi_i |y_{\xi_i}|^{2\gam_i} e^{-\varphi_i} e^{\delta_i} \phi^2 \, \d v_g +    \int_{\Sigma}h \phi \, \d v_g=  \sum_{i=1}^{p+q+N}  \|\tilde{\phi}_{i}\|^2_{\rL_i} +o(1)\rightarrow 0,
		\end{align*}
		where we applied that $\tilde{\phi}_i \to  0$ strongly in $\rL_{\xi_i}$. This yields a contradiction.
	\end{proof}
	
	\subsection{The non-linear problem}
	In this part, for fixed $\xi=(\xi_1, \ldots, \xi_{p+q+N})$ and $\eps>0$, we will find a solution $\phi_{\xi,\eps}$ of \eqref{eq:main_L2} by the fixed point theorem and reduce the problem to a finite-dimensional one.

	As $\eps\to 0$, the remainder term $\cR$ defined in \eqref{def:R} is small in $L^s$-space for any $s>1$ sufficiently close to $1$.

	\begin{lemma}\label{lem:est_error}
		For any $\iota\in (0,1)$, 
		there exists $s_0>1$ such that, for any $s\in(1,s_0)$, we have as $\eps\to 0$,
		$
		\left\|\mathcal{R}\right\|_s=\cO(\eps^{\frac{ \iota }{1+\max\{ 0,\gam_*\}}} ), 
		$
		where $\gam_*=\max_{i=1,\ldots, p+q+N} \gam_i$. 	\end{lemma}
	\begin{proof}
		By Lemma \ref{lem:diif_e^W_sum_e^u} and the H\"{o}lder's inequality, it holds 
		\begin{align*}
			\|\cR\|_s &= \Big\|\sum_{i=1}^{p+q+N} \Delta_g P\delta_i + \eps^2  Ke^{W} - \overline{  \eps^2  K e^{W}  }\Big \|_s\\
			&=  \Big \|\eps^2  Ke^{W}  -\sum_{i=1}^{p+q+N} \chi_i |y_{\xi_i}|^{2\gam_i} e^{-\varphi_i}e^{\delta_i} -\overline{\Big(\eps^2  Ke^{W}  -\sum_{i=1}^{p+q+N} \chi_i |y_{\xi_i}|^{2\gam_i} e^{-\varphi_i} e^{\delta_i} \Big)}\Big\|_s
			\\
			&\leq C \Big \| \eps^2  Ke^{\sum_{i=1}^{p+q+N} P\delta_i}  -\sum_{i=1}^{p+q+N} \chi_i |y_{\xi_i}|^{2\gam_i} e^{-\varphi_i} e^{\delta_i}  \Big\|_s \leq C(\eps^{\frac{ \iota }{1+\max\{ 0,\gam_*\}}}),
		\end{align*} 
		for some constant $C>0.$
	\end{proof}
	Next, we are going to show that the higher order linear operator $S(\phi)$ defined by \eqref{def:S} is bounded on $\oH$ and the operator norm goes to zero as $\eps\to  0$. 
	\begin{lemma}\label{lem:esti_S}
		For any $\iota\in (0,1)$, there exists $s_0> 1$  sufficiently close to $1$  such that for any $s,r \in (1, 2)$ with  $sr\in (1,s_0) $, as $\eps\rightarrow 0$
		\[
		\| S(\phi) \|_s =\cO(\eps^{\frac{ \iota }{1+\max\{0, \gam_*\}}} \| \phi \|  ), \quad \forall \phi\in \oH.
		\]
	\end{lemma}
	\begin{proof}
		By Lemma~\ref{lem:diif_e^W_sum_e^u}, the H\"{o}lder's inequality and the Moser-Trudinger inequality, we  derive that as $\eps\rightarrow 0$
		\begin{align*}
			\| S(\phi) \|_s &=\cO\Big( \Big\| \Big( \eps^2 Ke^{W}-\sum_{i=1}^{p+q+N} \chi_i |y_{\xi_i}|^{2\gam_i}e^{-\varphi_i} e^{\delta_i}\Big) \phi\Big\|_{s}\Big)\\
			&= \cO\Big( \Big\| \Big( \eps^2 Ke^{W}-\sum_{i=1}^{p+q+N} \chi_i |y_{\xi_i}|^{2\gam_i}e^{-\varphi_i} e^{\delta_i}\Big) \phi \Big\|_{s}\Big)\\
			&= O\Big( \Big\|\eps^2 K e^{W}-\sum_{i=1}^{p+q+N} \chi_i |y_{\xi_i}|^{2\gam_i}e^{-\varphi_i} e^{\delta_i}  \Big\|_{sr} \| \phi \|_{\frac{sr}{r-1}} \Big) \\
			&= \cO(\eps^{\frac{ \iota}{1+\max\{ 0, \gam_*\}}}  \|\phi\| ), 
		\end{align*}
		for any $s,r\in(1,2)$ with $sr<s_0$, where $s_0>1$ is from Lemma \ref{lem:diif_e^W_sum_e^u}. 
	\end{proof}

	Let 
	\begin{equation}
		\label{eq:def_F} F :\oH\to \R,\quad  u\mapsto \eps^2 K e^u
	\end{equation}
	
	The following lemma shows the asymptotic behavior of the nonlinear term $N(\phi)$ defined by \eqref{def:nonlinear}. 
	\begin{lemma}\label{lem:non_linear_term}
		There exist $c,\eps_0>0$ and  $s_0 > 1$ such that for any $s > 1$, $r > 1$ with $sr \in (1, s_0), \eps\in(0, \eps_0),$
		\begin{equation}\label{eq:3.4-1}
			\|N(\phi)\|_s = \cO\left( \eps^{\frac{1-sr}{(1+\gam_-)sr}}  \, e^{c\|\phi\|^2 } \|\phi\|^{2} \right), 
		\end{equation}
		and
		\begin{equation}\label{eq:3.4-2}
			\|N(\phi^1) - N(\phi^0)\|_s =\cO \left(\eps^{\frac{1-sr}{(1+\gam_-)sr}} \,  e^{c(\|\phi^0\|^2+\|\phi^1\|^2 ) } \|\phi^1 - \phi^0\|(\|\phi^1\| + \|\phi^0\|) \right),
		\end{equation}
		for  any  $\phi, \phi^1, \phi^0 \in \{ \phi \in \oH : \|\phi\| \leq 1 \}$, where $\gam_-=\min\{ 0,\min_{i=1,\ldots, p+q+N}\gam_i\}. $
	\end{lemma}
	\begin{proof}
		Firstly, we prove the estimate \eqref{eq:3.4-2}. The estimate \eqref{eq:3.4-1} then follows by taking  $\phi^0=0$ and $\phi^1=\phi$. 
		Observe that 
		\begin{align*}
			F(W+\phi^1)-F(W+\phi^0)- F'(W)(\phi^1-\phi^0) 
			&
			= \eps^2 K e^{W+\phi^1} -\eps^2 Ke^{W+\phi^0}-\eps^2 K e^{W}(\phi^1-\phi^0). 
		\end{align*}
		Applying the mean value theorem,  we have for any $s>1$
		\begin{align*}
			&\quad \|F(W+\phi^1)-F(W+\phi^0)- F'( W)(\phi^1-\phi^0)\|_{s}\\
			&
			= \|(F'(W+ \theta\phi^1+(1-\theta)\phi^0) -F'(W)) (\phi^1-\phi^0)\|_{s}\\
			&=\|F^{\prime\prime}(W+ \theta'\theta\phi^1+\theta' (1-\theta)\phi^0) (\theta\phi^1+(1-\theta)\phi^0) (\phi^1-\phi^0)\|_{s},
		\end{align*}
		for some $\theta, \theta'\in (0,1).$
		Via the H\"{o}lder's inequality, the Sobolev's inequality, and the  Moser-Trudinger inequality, we deduce  that 
		\begin{align*}
			&\quad	\|F^{\prime\prime}(W+ \theta'\theta\phi^1+\theta' (1-\theta)\phi^0) (\theta\phi^1+(1-\theta)\phi^0) (\phi^1-\phi^0)\|_s \\
			&\leq C \sum_{i=0}^1
			\left(  \int_{\Sigma}
			\eps^{2s} K^s	e^{sW}(  e^{|\phi^0|+|\phi^1|}  |\phi^1-\phi^0||\phi^i| ) ^s \, \d v_g \right)^{\frac 1 s}  \\
			&\leq   C \sum_{i=0}^1 \left(\int_{\Sigma} \eps^{2sr}K^{sr}e^{sr W} \, \d v_g\right)^{\frac{1}{sr}} \left( \int_{\Sigma}   e^{st' (|\phi^0|+|\phi^1)} \, \d v_g\right)^{\frac{1}{st'}} \left(\int_{\Sigma} |\phi^1-\phi^0|^{st} |\phi^i|^{st} \, \d v_g\right)^{\frac{1}{st}}  \\
			&\leq C \sum_{i=0}^1 \| \eps ^2 Ke^{ W} \|_{sr} e^{\frac{st'}{8\pi}(\|\phi^1_1\|^2+\|\phi^0_1\|^2)} \|\phi^1-\phi^0\| \|\phi^i\|,
		\end{align*}
		where $ r,t',t \in (1, +\infty), {  \frac{1}{r}+\frac{1}{t'}+\frac{1}{t}=1}$.
		Applying Lemma~\ref{lem:diif_e^W_sum_e^u}, we deduce that 
		\begin{align*}
			\|\eps^2 Ke^{W}\|_{sr}&\leq \Big\| \eps^2 Ke^{W} -\sum_{i=1}^{p+q+N} \chi_i |y_{\xi_i}|^{2\gam_i}e^{-\varphi_i} e^{\delta_i} \Big\|_{sr}+ \sum_{i=1}^{p+q+N} \Big\|  \chi_i |y_{\xi_i}|^{2\gam_i}e^{-\varphi_i} e^{\delta_i}   \Big\|_{sr}\\
			&\leq\cO( \eps^{\frac{ 2-sr}{(1+\gam_*)sr}}+ \eps^{\frac{ 2+2 sr \gam_-}{(1+\gam_*)sr}} + \eps^2  + \eps^{\frac{1-sr}{(1+\gam_-)sr}})=\cO( \eps^{\frac{1-sr}{(1+\gam_-)sr}}) . 
		\end{align*}
		By taking $c=\frac{st'}{8\pi}$, the estimate \eqref{eq:3.4-2} is concluded. 
	\end{proof}
	For fixed  $\xi\in M_{\sigma,\xi^*}$, we will find $\phi_{\xi,\eps}$ to solve the problem~\eqref{eq:main_L2} in $K^{\perp}_{\xi}$, i.e.  
	\begin{equation}
		\label{eq:toda_inf_dim} 
		\phi=\Pi_{\xi}^{\perp}\circ \cL^{-1}(S(\phi)+N(\phi)+\cR)
	\end{equation}
	for $\phi \in K_{\xi}^{\perp}$, where $\Pi_{\xi}^\perp: \oH\to K_{\xi}^\perp$ is the orthogonal projection from $\oH$ onto $K^\perp_{\xi}.$
	\begin{theorem}~\label{thm2_p} Let $\Dc$ be a compact subset of $M_{\sigma,\xi^*}$,
		and $\xi=(\xi_1,\ldots,\xi_{p+q+N})=(\xi^0,Q_1)\in \Dc$.  Assume that one of the following conditions holds:
		\begin{equation}
			\label{condi:i-ii} \text{ (i) } p+q=0; \quad  \text{ or } \quad \text{ (ii) } p+q\neq 0, \gam_*<1.
		\end{equation}
		There exist $s_0>1, \eps_0>0$ and $R>0$ (uniformly in $\xi$) such that for any $s\in (1, s_0)$ and  any $\eps\in (0,\eps_0)$ there is a unique $\phi_{\xi,\eps}\in K^{\perp}_{\xi}$ solves~\eqref{eq:toda_inf_dim}
		satisfying that 
		\[ \|\phi_{\xi,\eps}\|\leq   R \eps^{\iota_0 } |\ln \eps|, \]
		where $\iota_0=
		\frac  3 4  $ when   $ p+q=0$; $  \iota_0=\frac 1 2 ( \frac 1 2 + \frac 1 {1+\max\{0, \gam_*\}})  $ when $   p+q\neq 0$. 
		
		Furthermore, 
		the map $\xi^0\mapsto \phi_{\xi,\eps}$ is $C^1$-map with respect to $\xi^0\in M_{\sigma,\xi^*}^0$.
	\end{theorem} 
	\begin{proof}
		Given that $\xi\in \Dc$, we define the linear operator 
		\[ \cT_{\xi,\eps}(\phi):=\Pi_{\xi}^{\perp}\circ \cL^{-1}(S(\phi)+N(\phi)+\cR)\quad  \text{ in } K_\xi^\perp. \]
		
		For any $\phi\in K_{\xi}^{\perp}$, 
		Lemma \ref{lem:invertible}-\ref{lem:non_linear_term} yield that   for some constants $s_0>1$, $C_0, C, c>0$
		it holds 
		\begin{eqnarray*}
			\|	\cT_{\xi,\eps}(\phi)\|&\leq& C |\ln  \eps|\| S(\phi)+N(\phi)+\cR\|_s\\
			&\leq & C_0 |\ln \eps| ( \eps^{\iota/(1+\max\{ 0,\gam_*\})} \|\phi\|+ \eps^{ \frac{1-s r'}{(1+\gam_-)s r'}}  e^{c\|\phi\|^2 } \|\phi\|^{2} + \eps^{\iota/(1+\max\{ 0, \gam_*\})} ),
		\end{eqnarray*}
		for any $s, r'\in(1,2)$   sufficiently close to $1.$
		In case~(i) we take $\iota=\frac 3 4 $, while in case~(ii) $\iota=\frac{  3 +\max\{ 0, \gam_*\}}{4}\in(0,1)$.
		
		We fix $R=3C_0$ and assume that  $sr'\in (1, \min\{ s_0,  \frac{2}{1-\gam_-}, \frac 3 2  \})$. It follows that 
		\[ 	\|	\cT_{\xi,\eps}(\phi)\|\leq  C_0 |\ln \eps| ( \eps^{\iota_0} \|\phi\|+ \eps^{ \frac{1-s r'}{(1+\gam_-)s r'}}  e^{c\|\phi\|^2 } \|\phi\|^{2} + \eps^{\iota_0} ). \]
		Then, we fix  $s, r'>1$  sufficiently close to $1$ such that 
		\[\frac{1-s r'}{(1+\gam_-)s r'} +\iota_0  >0. \]
		There exists $\eps_1>0$ sufficiently small  such that for any $\eps\in (0,\eps_1)$ we have 
		\[ \max\left\{3C_0 \eps^{\iota_0 }|\ln \eps|, (3C_0)^2 e^{c} \eps^{\frac{1-s r'}{(1+\gam_-)s r'}  +\iota_0 } |\ln \eps|^2 , 3 C_0  \eps^{\iota_0 }|\ln \eps| \right\}\leq 1. \]
		Therefore, for any   $\phi\in  \{ \phi\in K_{\xi}^{\perp}: \|\phi\|\leq R \eps^{\iota_0 }|\ln \eps| \}$ it follows that 
		$$\|\cT_{\xi,\eps}(\phi)\|\leq R \eps^{\iota_0}|\ln \eps|  .$$
		Furthermore, for any  $\phi^0,\phi^1\in \{ \phi\in K_{\xi}^{\perp}: \|\phi\|\leq R\eps^{\iota_0}|\ln \eps|\}$,
		by Lemmas~\ref{lem:esti_S} and \ref{lem:non_linear_term}, we obtain:
		\begin{align*}
			\|\cT_{\xi,\eps}(\phi^1)-\cT_{\xi,\eps}(\phi^0)\|&\leq C |\ln  \eps| \cdot \| S(\phi^1-\phi^0)+N(\phi^1)-N(\phi^0)\|_s\\
			&\leq C_1|\ln \eps|\left( \eps^{\iota_0}\|\phi^1-\phi^0\|  + R\eps^{\frac{1-sr'}{(1+\gam_-)sr'}+\iota_0}|\ln \eps|\|\phi^1 - \phi^0\|\right). 
		\end{align*}
		Since $ \frac{1-sr'}{sr'}+\iota_0 >0$, there exists $\eps_2\in (0, \eps_1)$ such that for any $\eps\in(0,\eps_2)$, 
		\[ \max\left\{ C_1 \eps^{\iota_0} |\ln \eps| , C_1R\eps^{\frac{1-sr'}{(1+\gam_-)sr'}+\iota_0 }|\ln \eps|^2 \right\}\leq \frac 1 4. \]
		Letting 
		$\eps_0:=\eps_2$, 
		we conclude that for any   $\cT_{\xi,\eps}$ is contract mapping on $\{ \phi\in K_{\xi}^{\perp}: \|\phi\|\leq R\eps^{\iota_0}|\ln \eps|\}$. 
		By the Banach fixed-point theorem, we deduce that there exists a unique solution  $\phi_{\xi,\eps}$ in $\{ \phi\in K_{\xi}^{\perp}: \|\phi\|\leq R\eps^{\iota_0}|\ln \eps|\}$ such that  $\phi_{\xi,\eps}$ solves the problem~\eqref{eq:toda_inf_dim}.
		
		Denote that $\Dc^0:= \{ \xi^0\in M^0_{\sigma,\xi^*}: (\xi^0,Q_1)\in \Dc\}$. 
		Let $F(u)$ be defined by~\eqref{eq:def_F}.
		We define a function $\Phi: \Dc^0\times K^\perp_{\xi}\rightarrow K^\perp_{\xi}$, $(\xi^0,\phi)\mapsto \phi+\Pi_{\xi}^{\perp}(W-i^*\circ F(W+\Pi_{\xi}^{\perp}(\phi) ) ).$
		It is clear that 
		$\Phi(\xi^0,\phi_{\xi,\eps})=0$. For    any $\psi\in K^\perp_{\xi}$, we have 
		$\frac{\partial \Phi}{\partial\phi}(\xi^0,\phi_{\xi,\eps})(\psi)
		=\psi- \Pi_{\xi}^{\perp}\circ i^*(F^{\prime}(W+\phi_{\xi,\eps})(\Pi_{\xi}^{\perp}\psi)).$
		
		Next, we will show that 
		$ \frac{\partial \Phi}{\partial\phi}(\xi^0,\phi_{\xi,\eps})$ is non-degenerate.
		For $\psi\in \oH$, we have 
		\begin{align*}
			\frac{\partial \Phi}{\partial \phi}(\xi^0,\phi_{\xi,\eps})(\psi)&= \Pi_{\xi}^\perp \Big( \psi -i^*\Big(\sum_{i=1}^{p+q+N} \chi_i e^{-\varphi_i} e^{\delta_i} \psi \Big)\Big)-\Pi_{\xi}^{\perp}\circ i^*\circ S (\Pi_{\xi}^{\perp}\psi)\\
			&\quad
			- \Pi_{\xi}^{\perp}\circ i^*\circ  (F^{\prime}(W+\phi_{\xi,\eps})-F^{\prime}(W))
			(\Pi_{\xi}^{\perp}\psi)\\
			&= \Pi_{\xi}^\perp \circ i^*\circ  \cL(\Pi_{\xi}^\perp \psi )-\Pi_{\xi}^{\perp}\circ i^*\circ S (\Pi_{\xi}^{\perp}\psi)\\
			&\quad
			- \Pi_{\xi}^{\perp}\circ i^*\circ  (F^{\prime}(W+\phi_{\xi,\eps})-F^{\prime}(W))
			(\Pi_{\xi}^{\perp}\psi). 
		\end{align*}
		By the mean value theorem, there exists $\theta\in (0,1)$ such that 
		\begin{equation*}
			\|(F^{\prime}(W+\phi_{\xi,\eps})-F^{\prime}(W))\Pi_{\xi}^{\perp}\psi \|_{s}= \|F^{\prime\prime}(W+\theta\phi_{\xi,\eps})( \phi_{\xi,\eps})(\Pi_{\xi}^{\perp}\psi)\|_s \leq C \eps^{\frac{1-sr}{(1+\gam_-)sr}}\|\phi_{\xi,\eps}\|  \|\Pi_{\xi}^{\perp}\psi\|.
		\end{equation*}

		By Lemma~\ref{lem:invertible} and Lemma~\ref{lem:esti_S}   we have  for some constant $c>0$
		\begin{align*}
			\quad \Big\|	\frac{\partial \Phi}{\partial\phi}(\xi^0,\phi_{\xi,\eps})(\psi)\Big\|
			&\geq c\|\cL(\Pi_{\xi}^{\perp}\psi)\| - \|S\|\|\Pi_{\xi}^{\perp}\psi\|- \|(F^{\prime}(W+\phi_{\xi,\eps})-F^{\prime}(W))\Pi_{\xi}^{\perp}\psi\|_s\\
			&\geq\frac{c}{|\ln \eps|}\|\Pi_{\xi}^{\perp}\psi\| - \mathcal{O}(\eps^{\iota_0}\|\Pi_{\xi}^{\perp}\psi\|)+\mathcal{O}(\eps^{\frac{1-sr}{(1+\gam_-)sr}} \|\phi_{\xi,\eps}\|\|\Pi_{\xi}^{\perp}\psi\|)\\
			&\geq\frac{c}{2|\ln \eps|}\|\psi\|,
		\end{align*}
		for $s, r>1$ sufficiently close to $1$. 
		Hence, we obtain that $\frac{\partial \Phi}{\partial\phi}(\xi^0,\phi_{\xi,\eps})$ is invertible in $K^\perp_{\xi}$ with 
		$$\Big \| \Big(\frac{\partial \Phi}{\partial\phi}(\xi,\phi_{\xi,\eps})\Big)^{-1}\Big\|\leq \frac 2 c |\ln  \eps|.$$
		By the implicit function theorem, we have 
		$\xi^0\mapsto \phi_{\xi,\eps}$ is $C^1$-differentiable. 
	\end{proof}

	\section{Reduced functional and its expansion}\label{sec:4}
	Via Theorem \ref{thm2_p}, the problem \eqref{eq:main_L2} is reduced to be a finite-dimensional one   with $\xi:=(\xi^0,Q_1)\in  \Xi_{p,q}^\prime\times \{ \xi^*_{p+q+1}\}\times \ldots \times \{ \xi^*_{p+q+N}\}$. 
	By substituting  $W+\phi_{\xi,\eps}$ into the energy functional $E_\eps$, we obtain the reduced functional $$\tilde{E}_\eps(\xi^0)= E_\eps(W+\phi_{\xi,\eps})
	= \frac 1 2 \int_\Sigma |\nabla(W+\phi_{\xi,\eps})|_g^2\, \d v_g+ \eps^2 \int_\Sigma Ke^{W+\phi_{\xi,\eps}}\, d v_g,$$ which is defined on $\Xi_{p,q}^\prime.$ Before studying the asymptotic behavior of $\tilde{E}_\eps$, we expand  on $E_\eps(W)$ as follows: 
	\begin{lemma}\label{lem:key_energy_app}
		Given $m\in \N$ and $(p,q,Q_1)\in \cI_m$, we assume that \eqref{def:d_i}-\eqref{def:lam_i} hold. Set $\rho_*:=4\pi n_m$. Then there exists $\eps_0>0$ such that, for any $\eps\in(0,\eps_0)$, the following expansion holds as $\eps\to 0$:
		\begin{equation*}
			E_\eps(W)= (3\ln  2)\rho_* -2\rho_* \ln  \eps +\sum_{i=1}^{p+q+N} 2(1+\gam_i)\varrho(\xi_i) \ln (1+\gam_i)  -\frac 1 2 \ff(\xi^0) +o(1).
		\end{equation*}
		and 	for $i=1,\ldots, p+q,$ and $j=1,\ldots,\ii(\xi_i)$
		$$\partial_{(\xi_i)_j}	E_{\eps}(W)
		=-\frac 1 2 \partial_{(\xi_i)_j} \ff (\xi^0)+o(1),$$ 
		which are convergent  in  $C(M^0_{\sigma,\xi^*})$ uniformly for any $\xi$ in any compact subset of $M^0_{\sigma,\xi^*}$. 
	\end{lemma}
	
	\begin{proof}
		Let $\Dc$ be a compact subset of $M_{\sigma,\xi^*}$ and $\xi=(\xi_1,\ldots, \xi_{p+q+N})=(\xi^0,Q_1)\in \Dc$. 
		The Dirichlet integral has the following asymptotic estimate: 
		\begin{align}
			&\quad 	\frac 1 2 \int_\Sigma |\nabla W|^2_g \, \d v_g   \label{eq:Dirichlet} 
			\\&= \frac 1 2 \sum_{i=1}^{p+q+N} \la P\delta_i, P\delta_i\ra  +\frac 1 2 \sum_{\mytop{i,j=1, \ldots, p+q+N}{i\neq j}} \la  P\delta_i, P\delta_j\ra\nonumber\\
			&= (-1+3\ln  2)\rho_* -2\rho_* \ln  \eps +\sum_{i=1}^{p+q+N} 2(1+\gam_i)\varrho(\xi_i) \ln (1+\gam_i)  -\frac 1 2 \ff(\xi^0)  +o(1),  \nonumber
		\end{align}
		by Lemma \ref{lem:innner_proj}. 
		It remains to calculate the nonlinear term.  	Using  Lemma~\ref{lem:proj_bubble} and Lemma~\ref{lem:diif_e^W_sum_e^u}, we deduce that for $s\in (1,s_0)$ 
		\begin{align}
			&	\quad\eps^2 \int_{\Sigma} Ke^{W}\, \d v_g \label{eq:total_mass} \\
			&=  \mathcal{O}\Big( \Big| \int_{\Sigma} \eps^2 K  e^{W}-\sum_{i=1}^{p+q+N}  \chi_i  |y_{\xi_i}|^{2\gam_i}e^{-\varphi_i} e^{\delta_i} \, \d v_g\Big|\Big)   +  \sum_{i=1}^{p+q+N} \int_{\Sigma} \chi_i |y_{\xi_i}|^{2\gam_i} e^{-\varphi_i} e^{\delta_i} \, \d v_g \nonumber  \\
			&= \mathcal{O}\Big(\Big\|  \eps^2 K  e^{W}-\sum_{i=1}^{p+q+N}  \chi_i  |y_{\xi_i}|^{2\gam_i} e^{-\varphi_i} e^{\delta_i}  \Big\|_{s}\Big)+ \sum_{i=1}^{p+q+N} (1+\gam_i)\varrho(\xi_i) +o(1)\nonumber\\
			&=\sum_{i=1}^{p+q+N} (1+\gam(\xi_i)) \varrho(\xi_i)+o(1)\nonumber
		\end{align}
		where  we applied that $\int_{\R^2}\frac{ 8(1+\gam)^2|y|^{2\gam}}{(1+|y|^{2(1+\gam)})^2} \, \d y=  (1+\gam)8\pi$ for all $\gam>-1$ and $|\Sigma|_g=1$. 
		Combining \eqref{eq:Dirichlet} and \eqref{eq:total_mass}, we  derive that as $\eps\to 0$
		\[     E_\eps(W)= (3\ln  2)\rho_* -2\rho_* \ln  \eps +\sum_{i=1}^{p+q+N} 2(1+\gam_i)\varrho(\xi_i) \ln (1+\gam_i)  -\frac 1 2 \ff(\xi^0) +o(1). \]
		For $i=1,\ldots, p+q$ and $j=1,\ldots,\ii(\xi_i)$, we have 
		\begin{align*}
			\partial_{(\xi_i)_j} E_\eps(W) &= \int_\Sigma (-\Delta_g W-\eps^2 K e^W) \partial_{(\xi_i)_j} W \, \d v_g\\
			&=\int_\Sigma \Big( \sum_{l=1}^{p+q+N}\chi_l |y_{\xi_l}|^{2\gam_l}e^{-\varphi_l} e^{\delta_l}  -  \eps^2 K e^W\Big ) \partial_{(\xi_i)_j} W \, \d v_g, 
		\end{align*}		
		in view of $\int_\Sigma \partial_{(\xi_i)_j} W\, \d v_g  =0.$
		Observe that 
		\begin{equation}\label{eq:lem3.5-0}
			\partial_{(\xi_i)_j}P\delta_l= \begin{cases}\left. 
				\bdelta_{il}	\partial_{x_j} P\delta^{\gam_i}_{\lam_i,x}\right|_{x=\xi_i} -	\frac 1 2 	( \lam_l\partial_{\lam_l} P\delta^{\gam_l}_{\lam_l, \xi_l})   \partial_{(\xi_i)_j} \ln  d_l(\xi_l)  & \text{ for } l=1,\ldots,p+q, \\
				-\frac 1 {2(1+\gam_l)}	( \lam_l \partial_{\lam_l} P\delta^{\gam_l}_{\lam_l, \xi_l}) \partial_{(\xi_i)_j} \ln  d_l(\xi_l)  & \text{ otherwise. } 
			\end{cases}
		\end{equation}
		Lemma \ref{lem:proj_1st} and Remark \ref{rk:B.1} yield that  for $\forall i=1,\ldots,p+q,j=1,\ldots, \ii(\xi_i), $
		\begin{equation}\label{eq:lem3.5-1}
			\lam_i^{-1}
			\partial_{x_j} P\delta^{\gam_i}_{\lam_i,x}|_{x=\xi_i} = PZ^j_i+\cO(\eps^{2}|\ln \eps|)= \chi_i Z^j_i  +\cO( \eps^{2}|\ln \eps| );  
		\end{equation}
		and  for $\forall i=1,\ldots,p+q+N, $
		\begin{equation}\label{eq:lem3.5-2}
			\lam_i \partial_{\lam_i} P\delta_i = PZ^0_j+\cO(\eps^2+ \eps^{\frac 2 {1+\max\{0,\gam_i\}}}|\ln \eps|)= \chi_i( 2+Z^0_i)  +\cO(\eps^2+ \eps^{\frac 2 {1+\max\{0,\gam_i\}}}|\ln \eps|). 
		\end{equation}
		Using the estimates \eqref{eq:lem3.5-1}, \eqref{eq:lem3.5-2} and Lemma~\ref{lem:diif_e^W_sum_e^u}, we obtain,    for any $s\in (1,s_0)$ 
		\begin{align*}
			&\quad \int_{\Sigma} \Big( \sum_{l=1}^{p+q+N}\chi_l |y_{\xi_l}|^{2\gam_l}e^{-\varphi_l} e^{\delta_l}  -  \eps^2 K e^W\Big) \partial_{(\xi_i)_j} W\, \d v_g\\
			&=
			\sum_{l=1}^{p+q}  \int_{\Sigma}   \chi_l |y_{\xi_l}|^{2\gam_l}e^{-\varphi_l} e^{\delta_l} \lam_i Z^j_i  \, \d v_g -\eps^2 \int_{\Sigma} Ke^{W} \chi_i \lam_i Z^j_i \, \d v_g + 
			\mathcal{O}\Big(\Big\| \sum_{l=1}^{p+q}\chi_l e^{-\varphi_l} e^{\delta_l
			} - \eps^2   K e^W  \Big\| _{s}\Big)\\
			&=\sum_{l=1}^{p+q+N} \int_{\Sigma}  \chi_l |y_{\xi_l}|^{2\gam_l}e^{-\varphi_l} e^{\delta_l} \chi_i \lam_i Z^j_i \, \d v_g -\eps^2 \int_{\Sigma} K e^{W}\chi_i \lam_i Z^j_i \, \d v_g+o(1),
		\end{align*}
		as $\eps\rightarrow 0$. 
		For $j=1,\ldots,\ii(\xi_i)$, we have 
		\begin{align*}
			\int_{\Sigma} \chi_l |y_{\xi_l}|^{2\gam_l}e^{-\varphi_l} e^{\delta_l} \chi_i \lam_i Z^j_i \, \d v_g  
			&=
			\bdelta_{il}\lam_i^2 	\int_{B^{\xi_i}_{2r_0}}\chi^2\Big(\frac{|y|}{r_0}\Big)\frac{32\lam_i^2 y_j }{(1+\lam_i^2|y|^2)^3}  \, \d y  
			=0,
		\end{align*}
		in which the last equality applied the symmetric property of $ B^{\xi_i}_{2r_0}$. 
		It remains to  estimate  the integral $\eps^2 \int_{\Sigma} Ke^{W}\chi_i \lam_i Z^j_i \, \d v_g. $ Recall from \eqref{def:d_i} that
		$$	d_i(x)= \frac  1{ 8(1+\gam_i)^2} e^{(1+\gam_i) \varrho(\xi_i)H^g(x, \xi_i) +\sum_{\mytop{j=1,\ldots, p+q+N}{j\neq i}}(1+\gam_j)\varrho(\xi_j) G^g(x,\xi_j) +\ln  K_i(x)}.  $$       
		By Lemma~\ref{lem:proj_bubble},  it follows that  as $\eps\rightarrow 0$
		\begin{align*}
			\eps^2 \int_{\Sigma} 
			K e^{W}\chi_i \lam_iZ^j_i   \, \d v_g &=                                                                                                                                                                                                                                                                                                                                                                                                                                                                                                                                                                                                                                                                                                                                                                                                                                                                                                                                                                                                                                                                                                                                                                                                                                                                                                                                                                                                                                                                                                                                                                                                                                                                                                                                                                                                                                                                                                                                                                                                                                                                                                                                                                                                                                                                                                                                  \eps^2  \int_{\Sigma} \chi_i e^{\sum_{l=1}^{p+q+N}P\delta_l +\ln  K_i }\frac{4\lam_i^2  (y_{\xi_i})_j}{1+\lam_i^2 |y_{\xi_i}|^2} \, \d v_g \\
			&=    \int_{ \lam_iB^{\xi_i}_{r_0}} e^{\hat{\varphi}_i(y/\lam_i)} \frac{ 8 d_i\circ y_{\xi_i}^{-1}(y/\lam_i)}{d_i} \frac{ 4 \lam_i y_j}{(1+|y|^2)^3 } \, \d y + \cO(\lam_i^{-1})
			\\
			&=\frac 1 2 \partial_{(\xi_i)_j} \ff(\xi^0)+\cO(\lam_i^{-1}),
		\end{align*}
		where we applied the symmetric property of $\lam_i B^{\xi_i}_{r_0}$, \eqref{varphixi} and $\int_{\R^2} \frac{4y_j^2}{(1+|y|^2)^3} \, \d y=\pi.$ 
		
		Consequently, we prove that as $\eps\rightarrow 0$, 
		$ \partial_{(\xi_i)_j}E_{\eps}(W) =-\frac 1 2 \partial_{(\xi_i)_j} \ff(\xi^0)+o(1).  $
	\end{proof}
	Subsequently, we expand the reduced functional 
	$\tilde{E}_{\eps}: M^0_{\sigma,\xi^*}\rightarrow \R, \xi\mapsto E_{\eps}(W+\phi)$,  where $M^0_{\sigma,\xi^*}:=\{ \xi^0: (\xi^0, Q_1)\in M_{\sigma,\xi^*}\}$ and $\phi:=\phi_{\xi,\eps}$ is given by Theorem~\ref{thm2_p}. 
	\begin{theorem}
		\label{thm:expansion_E_reduced}
		Given  $m\in \N$ and $(p,q,Q_1)\in \cI_m$, we assume that~\eqref{def:d_i}-\eqref{def:lam_i} and \eqref{condi:i-ii} hold. Denote $\rho_*=4\pi n_m.$ Then, there exists $\eps_0>0$ such that   the expansion holds for any $\eps\in(0,\eps_0)$
		\begin{equation}\label{expansion_E_re}
			\tilde{E}_{\eps}(\xi^0)=(3\ln  2)\rho_* -2\rho_* \ln  \eps +\sum_{i=1}^{p+q+N} 2(1+\gam_i)\varrho(\xi_i) \ln (1+\gam_i)  -\frac 1 2 \ff(\xi^0) +o(1), 
		\end{equation}
		and  for any $i=1,\ldots, p+q$ and $j=1,\ldots,\ii(\xi_i)$
		\begin{equation}
			\label{eq:expansion_E_re-1st}  	\partial_{(\xi_i)_j} \tilde{E}_{\eps}(\xi^0)
			=-\frac 1 2 \partial_{(\xi_i)_j} \ff(\xi^0)+o(1),
		\end{equation}
		which are convergent  in  $C(M^0_{\sigma,\xi^*})$ uniformly for any $\xi^0$ in a compact subset of $M^0_{\sigma,\xi^*}$. 
	\end{theorem}
	\begin{proof}
		By Theorem~\ref{thm2_p}, the error term $\phi\in K_{\xi}^{\perp}$ satisfies, for $s$ sufficiently close to $1$,
		$
		\|\phi\|\le R\,\eps^{\iota_0}\,|\ln\eps|.
		$
		Thus, by Lemma~\ref{lem:innner_proj}, we obtain that
		\begin{align*}
			&\quad \tilde{E}_{\eps}(\xi^0)= E_{\eps}(W)+ \frac 1 2 \la \phi, \phi\ra + \sum_{i=1}^{p+q+N}\la P\delta_i, \phi\ra -\eps^2\int_\Sigma K  e^W( e^\phi-1-\phi)  \, \d v_g - \eps^2 \int_{\Sigma}  K e^W \phi \, \d v_g\\
			&= E_\eps(W)-\eps^2\int_\Sigma K e^W( e^\phi-1-\phi)  \, \d v_g -  \int_{\Sigma}  \Big( \eps^2 K e^W - \sum_{i=1}^{p+q+N}  \chi_i e^{-\varphi_i} |y_{\xi_i} |^{2\gam_i}  e^{\delta_i} \Big)\phi \, \d v_g+o(1). 
		\end{align*}
		For  $s,r, r'>1$ sufficiently close to $1$, Lemma 		\ref{lem:diif_e^W_sum_e^u} and  Lemma \ref{lem:non_linear_term} together with the H\"{o}lder's inequality yield  that 
		\begin{align*}
			\Big| \int_\Sigma \Big( \sum_{i=1}^{p+q+N}  \chi_i e^{-\varphi_i} |y_{\xi_i} |^{2\gam_i}  e^{\delta_i} -\eps ^2 Ke^{W}\Big) \phi \, \d v_g\Big| &\leq \Big\|   \sum_{i=1}^{p+q+N}  \chi_i e^{-\varphi_i} |y_{\xi_i} |^{2\gam_i}  e^{\delta_i} -\eps ^2 Ke^{W}  \Big\|_{s}\|\phi\|_{\frac{s}{s-1}}\\
			&\leq  \mathcal{O}\left(\eps^{\iota_0}\|\phi\|\right)=o(1),\\
			\Big |  \eps^2 \int_{\Sigma} Ke^{W}( e^{\phi}  -1 -  \phi) \, \d v_g\Big| &\leq  \mathcal{O}\Big( \eps^{\frac{1-sr'}{(1+\gam_-)sr'}} e^{c \|\phi\|^2} \|\phi\|^2 \Big) =o(1).
		\end{align*}
		Consequently,  we deduce  that
		$ \tilde{E}_{\eps}(\xi^0)=E_{\eps}(W)+o(1). $
		Immediately, Lemma~\ref{lem:key_energy_app} deduces  \eqref{expansion_E_re}, which holds in $C(M^0_{\sigma,\xi^*})$ and uniformly for  $\xi^0$ in any compact subset of $M^0_{\sigma,\xi^*}$.

		Using the result in  Theorem~\ref{thm2_p}, there exists $\{c^{\eps}_{i,j}\in\R: i=1,\ldots,p+q, j=1,\ldots,\ii(\xi_i) \}$ such that
		\begin{equation}\label{eq:solution_inf_dim}
			W+\phi -i^*(F(W+\phi ))=
			\sum_{\mytop{i=1,\ldots, p+q}{j=1,\ldots,\ii(\xi_i)}} c^{\eps}_{i,j} PZ^j_i,
		\end{equation}
		where $F(h):= \eps^2 K e^h$ for all $h\in \oH. $
		\begin{claim} \label{claim:1} For the coefficients, it holds 
			$\sum_{\mytop{i=1,\ldots, p+q}{j=1,\ldots,\ii(\xi_i)}}|c^\eps_{i,j}|=\cO( \eps ) $, as $\eps\to 0.$
		\end{claim}
		Applying \eqref{eq:solution_inf_dim}, we derive that for $i=1,\ldots,p+q, j=1,\ldots,\ii(\xi_{i})$
		\begin{align*}
			\sum_{\mytop{i'=1,\ldots,p+q}{j'=1,\ldots,\ii(\xi_{i'})}}  c^{\eps}_{i',j'}\lan PZ^{j'}_{i'}, PZ^j_i\ran &=\lan 		W+\phi -i^*(F(W+\phi )) , PZ^j_i\ran \\
			&= \lan 	W-i^*(  F(W)),PZ^j_i\ran+ \mathcal{O}( |\lan\phi, PZ^j_i\ran|  )\\
			&\quad  -\int_{\Sigma} (F(W + \phi) - F(W))PZ^j_i \, \d v_g. 
		\end{align*}
		From the proof of  Lemma~\ref{lem:key_energy_app},  for $i=1,\ldots, p+q$ we have 
		\begin{align*}
			\lan 	W-i^*(F(W)), PZ^j_i\ran&= \int_{\Sigma}\Big(\sum_{l=1}^{p+q+N}  \chi_l e^{-\varphi_l}|y_{\xi_l}|^{2\gam_l}e^{\delta_l} - \eps^2 Ke^{W}\Big)PZ^j_i \, \d v_g \\
			&= -\frac 1 {2\lam_i}\partial_{(\xi_i)_j}\ff(\xi^0)+ o(\lam_i^{-1}). 
		\end{align*}
		Combining Lemmas~\ref{lem:non_linear_term} and \ref{lem:diif_e^W_sum_e^u} with Remark~\ref{rk:B.1} and \eqref{eq:total_mass}, we infer that
		\begin{align*}
			&	\quad \int_{\Sigma} (F(W+\phi)-F(W))PZ^j_i \, \d v_g\\
			&=  \eps^2 \int_{\Sigma} Ke^{W}\phi (\chi_iZ^j_i+\cO(\eps^2|\ln \eps|))) \, \d v_g+ \mathcal{O}\left(  \eps^{\frac{ 1 -sr}{(1+\gam_-)sr}}\|\phi\|^2 \|PZ^j_i\|_{\frac s{s-1}} \right)\\
			&=  \eps^2 \int_{\Sigma} Ke^{W} \phi \chi_iZ^j_i \, \d v_g    + \cO(  \eps^{\frac{ 1 -sr}{(1+\gam_-)sr}  } \|\phi\|^2  )\\
			&=  \int_{\Sigma} \Big| \eps^2  Ke^W -\sum_{l=1}^{p+q+N} \chi_l e^{-\varphi_l}|y_{\xi_l}|^{2\gam_l} e^{\delta_l}\Big| |\phi| |PZ^i_j| \, \d v_g + \la  PZ^i_j,\phi\ra  +\cO(   \eps^{\frac{ 1 -sr}{(1+\gam_-)sr}  } \|\phi\|^2 ) \\
			&=   \cO( \eps^{ \iota_0 }\|\phi\|+  \eps^{\frac{ 1 -sr}{(1+\gam_-)sr}  } \|\phi\|^2 ) =\cO(\eps^{\iota_0+\frac 1 2 }),
		\end{align*}		
		for $s,r>1$ sufficiently close to $1$, 	in view of $\iota_0>\frac 1 2$, $\la PZ^j_i,\phi\ra=0$ and $PZ^j_i=\cO(1)$ for $i=1,\ldots, p+q, j=1,\ldots, \ii(\xi_i)$. 
		Using Lemma \ref{lem4_p}, we derive that 
		\begin{equation}
			\sum_{\mytop{i=1,\ldots,p+q}{j=1,\ldots,\ii(\xi_i)}} \frac{4\varrho(\xi_i)}{3} |c^{\eps}_{i,j} | =\cO\Big(\sum_{\mytop{i=1,\ldots,p+q}{j=1,\ldots,\ii(\xi_i)}} \frac{4\varrho(\xi_i)}{3} |c^{\eps}_{i,j} |  ( \eps^2+ \eps^{\frac 2 {1+\max\{0,\gam_*\}}}|\ln \eps| )+ \eps^{\iota_0+\frac  1 2  }+\eps \Big). 
		\end{equation}
	Since $\iota_0>\frac12$, we proved Claim \ref{claim:1}. 
		
		For $s', s>1$ sufficiently close to $1$, we have  for $ i'=1,\ldots, p+q,j'=1,\ldots,\ii(\xi_{i'})$
		\begin{align}
			\quad \lan PZ^{j'}_{i'}, \partial_{\left(\xi_{i}\right)_{j}} \phi \ran&=\partial_{(\xi_i)_j}\lan PZ^{j'}_{i'}, \phi\ran- \lan \partial_{\left(\xi_{i}\right)_{j}} PZ^{j'}_{i'}, \phi \ran \nonumber\\
			&=\int_{\Sigma} \partial_{\left(\xi_{i}\right)_{j}}( -\chi_{i'}e^{-\varphi_{i'}} e^{\delta_{i'}} Z^{j'}_{i'} )\phi \, \d v_g=\mathcal{O}\left( \|\partial_{\left(\xi_{i}\right)_{j}}( -\chi_{i'}e^{-\varphi_{i'}} e^{\delta_{i'}} Z^{j'}_{i'} )\|_{s'}\|\phi\|\right)\nonumber\\
			& =\mathcal{O}(\eps^{\frac{2-3s'}{s'}+ \iota_0 } |\ln \eps|)=o(\eps^{ -\frac 1 2 }), \label{eq:pz_pa_phi_l}
		\end{align}
		where we applied the assumption $\lan PZ^{j'}_{i'},\phi\ran= 0.$ By \eqref{eq:solution_inf_dim} and Claim \ref{claim:1}, it  follows that 
		\begin{align*}
			\partial_{(\xi_i)_j}\tilde{E}_{\eps}(\xi^0) &= \partial_{(\xi_i)_j}E_{\eps}(W) + \int_{\Sigma} (- \Delta_g \phi- \eps^2  K ( e^{W+ \phi} - e^{W}) )  \partial_{(\xi_i)_j} W  \, \d v_g\\
			&\quad +\sum_{\mytop{i'=1,\ldots, p+q }{j'=1,\ldots, \ii(\xi_i)}} c^\eps_{i',j'} \la PZ^{j'}_{i'},\partial_{(\xi_i)_j} \phi \ra  \\
			&= \partial_{(\xi_i)_j}E_{\eps}(W) + \int_{\Sigma} (- \Delta_g \phi- \eps^2  K ( e^{W+ \phi} - e^{W}) )  \partial_{(\xi_i)_j} W  \, \d v_g+o(1),
		\end{align*}
		where we applied that $\phi\in K_{\xi}^\perp.$
		The estimates \eqref{eq:lem3.5-0}-\eqref{eq:lem3.5-2} imply that
		$ \partial_{(\xi_i)_j}  W=\lam_i PZ^j_i+\mathcal{O}(1).$
		Providing that  $\lan \phi, PZ^j_i\ran =0, $ we have
		\begin{align*}
			\int_{\Sigma} (-\Delta_g\phi)\partial_{(\xi_i)_j} W\, \d v_g &=\int_{\Sigma}\lan\nabla\phi,(\lam_i \nabla PZ^j_i+\mathcal{O}(1))\ran_g  \, \d v_g =\lam_i \lan \phi, PZ^j_i\ran+ \mathcal{O}\left(\|\phi\| \right)=o(1).
		\end{align*}
		Moreover, 
		Lemma \ref{lem:non_linear_term} together with Remark \ref{rk:B.1} yields  for $s, r>1$ sufficiently close to $1$
		\begin{align*}
			&\quad	\int_{\Sigma}	(F(W+ \phi) - F(W))  \partial_{(\xi_i)_j}W\, \d v_g\\
			&=\eps^2 \int_{\Sigma} Ke^{W}\phi \partial_{(\xi_i)_j}W \, \d v_g
			+\mathcal{O}\left( \int_{\Sigma}
			|(F(W+\phi)-F(W)- F'(W)\phi )\partial_{(\xi_i)_j}W|  \, \d v_g\right)\\
			& =  \eps^2 \int_{\Sigma} Ke^{W}\phi ( \chi_i \lam_i PZ^j_i +\cO(1))\, \d v_g + \mathcal{O}\left( \lam_i  \eps^{\frac{ 1 -sr}{(1+\gam_-)sr}}\|\phi\|^2 \|PZ^j_i\|_{\frac s {s-1}} \right)\\
			&=\eps^2 \int_{\Sigma} Ke^{W}\phi \chi_i\lam_i Z^j_i \, \d v_g    +o(1),
		\end{align*}
		where $F(h)= \eps^2 K e^h$ for any $h\in \oH.$
		By Lemma~\ref{lem:diif_e^W_sum_e^u} and \eqref{def:lam_i}, for any $s, r\in (1,2)$ sufficiently close to $1$
		\begin{align*}
			&\quad 	\eps^2 \int_{\Sigma} Ke^{W}\phi\chi_i\lam_i Z^j_i \, \d v_g\\
			&= \int_{\Sigma}\chi_i e^{-\varphi_i} e^{\delta_i}\lam_i Z^j_i \phi \, \d v_g + \mathcal{O}\Big(\lam_i  \Big\| \eps^2  Ke^{W}-\sum_{l=1}^{p+q+N}  \chi_l e^{-\varphi_l}|y_{\xi_l}|^{2\gam_l} e^{\delta_l} \Big\|_{sr} \| \chi_i Z^j_i\|_{\frac{s}{s-1}}\|\phi\|\Big)\\
			&=\lam_i \lan PZ^j_i, \phi\ran +\cO( \eps^{2\iota_0-1} |\ln \eps|)=o(1),
		\end{align*}
		in view of $\iota_0>\frac  1 2. $
		Combining all the estimates above with Lemma \ref{lem:key_energy_app}, we prove that  as $\eps\rightarrow 0$
		\[ \partial_{(\xi_i)_j} \tilde{E}_{\eps}(\xi^0)=\partial_{(\xi_i)_j} E_{\eps}(W)+o(1)=-\frac 1 2  \partial_{(\xi_i)_j}\ff(\xi^0)+o(1). \]
	\end{proof}

	\section{Proof of  main results}\label{sec:5}
	Now, we are ready to prove the main results. 
	
	\begin{proof}[Proof of Theorem~\ref{thm:main0}]
		By Theorem \ref{thm2_p},  there exists $\eps_0>0$ sufficiently small  such that for any $\eps\in (0, \eps_0)$ and $\xi=(\xi^*_1,\ldots,\xi_N^*)\subset Q^N$, there exists  $\phi _{\xi,\eps}$ such that 
		\[  \phi_{\xi,\eps} =\Pi_{\xi}^\perp ( \cL^{-1}( S(\phi_{\xi,\eps})+N(\phi_{\xi,\eps})+\cR))\]
		with $\|\phi\|\leq R \eps^{\iota_0 }|\ln \eps|$, where $R>0$ is a large constant.  Since $p+q=0$ in this case, the space $K_\xi$ is trivial and hence $\Pi_\xi^\perp=\mathrm{Id}$.
		The solution $u_\eps:=\sum_{i=1}^N P\delta_i +\phi_{\xi,\eps}$ we constructed is a   solution of \eqref{eq:main_L} with respect to parameter $\eps$. 
		Denote  $\rho^\eps := \eps^2  \int_{\Sigma} K e^{u_{\eps}} \, \d v_g.$  Using the inequality  $|e^{s}-1| \leqslant e^{|s|}|s|$ for any $s \in \mathbb{R}$, for $q>1$ sufficiently close to $1$, we have 
		\begin{align*}
			\rho^\eps&=\eps^2 \int_{\Sigma} Ke^{W}\, \d v_g+ \mathcal{O}\Big(  \eps^2  \int_{\Sigma} Ke^{W}|\phi_{\xi,\eps}| \, \d v_g\Big) \stackrel{\eqref{eq:total_mass}}{=} \sum_{i=1}^{p+q+N}(1+\gam_i) \varrho(\xi_i)+\mathcal{O}(\|\eps ^2 Ke^{W}\|_{q}\|\phi_{\xi,\eps}\|)\\
			&
			=\sum_{i=1}^{N}(1+\gam(\xi_i^*)) \varrho(\xi^*_i) +o(1).
		\end{align*}
		For any $\Psi \in C(\Sigma),$
		by Lemma~\ref{lem:diif_e^W_sum_e^u}, we derive that  
		\begin{align*}
			\rho^\eps\int_{\Sigma} \frac{K \mathrm{e}^{u_{\eps}} } { \int_{\Sigma}K \mathrm{e}^{u_{\eps}} \, \d v_g} \Psi \, \d v_g &= \eps^2  \int_{\Sigma} K \mathrm{e}^{u_{\eps}}\Psi \, \d v_g
			=
			\sum_{i=1}^{N} \int_{\Sigma} \chi_{\xi_i^{*}} e^{-\varphi_{\xi^*_i}}|y_{\xi_i^*}|^{2\gam_i}e^{\delta^{\gam_i}_{\lam_i, \xi^*_i}}\Psi \, \d v_g +o(1) \\
			&= \sum_{i=1}^{N} (1+\gam(\xi_i^*)) \varrho(\xi^*_i) \Psi\left(\xi^*_i\right)+o(1),
		\end{align*} 
		as $\eps\rightarrow 0$.  Consequently, $u_\eps$ is a sequence of solutions of \eqref{eq:mf} with parameter $\rho^\eps\to \sum_{i=1}^N (1+\gam(\xi^*_i))\varrho(\xi_i^*)=\rho_*$, which blows up at $\xi_1^*,\ldots, \xi_{p+q+N}^*$ and satisfies  
		\begin{equation*}
			\rho^\eps\frac{ Ve^{v_{\eps}}} { \int_{\Sigma} Ve^{v_{\eps}}  \, \d v_g} \, \d v_g  \stackrel{*}{\rightharpoonup}
			\sum_{i=1}^{N} (1+\gam(\xi^*_i))\varrho(\xi^*_i)\bm{\delta}_{\xi^*_i},
		\end{equation*}
		where $\bm{\delta}_x$ is the Dirac mass on $\Sigma$ concentrated at $x\in \Sigma$ and 	 $v_\eps= u_\eps-h_Q(x).$ Therefore, 
		the proof is concluded.
		
		From the proof of  Lemma \ref{lem:key_energy_app} and Theorem~\ref{thm:expansion_E_reduced},  we have the following estimate for the energy level of the blow-up solutions $u_\eps$,  as $\eps\to 0$
		\begin{align*}
			J_{\rho^\eps}(u_\eps)&= \frac 1 2 \int_\Sigma |\nabla u_\eps|^2_g\, \d v_g - \rho_\eps \ln  \int_\Sigma  K e^{u_\eps}\\
			&=  \frac 1 2 \int_\Sigma |\nabla u_\eps|^2_g\, \d v_g   -  \Big( \eps^2 \int_\Sigma K e^{u_\eps} \, \d v_g\Big )\ln   \Big(\eps^2  \int_\Sigma  K e^{u_\eps} \, \d v_g\Big)+ 2 \Big( \eps^2 \int_\Sigma K e^{u_\eps} \, \d v_g\Big )\ln \eps\\
			&= (1+3\ln  2)\rho_*+\sum_{i=1}^{N} 2 (1+\gam(\xi_i^*))\varrho(\xi_i^*) \ln (1+\gam(\xi_i^*))  -\rho_* \ln  \rho_*+o(1)\\
			&\to\rho_* -\rho_* \ln  \Big(\frac {\rho_*}{8}\Big)+ 
			\sum_{i=1}^{N} 2 (1+\gam(\xi_i^*))\varrho(\xi^*_i) \ln (1+\gam(\xi_i^*)).
		\end{align*}

	\end{proof}
	
	Next, we will show that  for $\eps>0$ sufficiently small 
	$\xi^0$ being a critical point of $\tilde{E}_{\eps}$  is equivalent to  $W+\phi_{\xi,\eps}$ solving \eqref{eq:main_L2} for $\xi=(\xi^0,Q_1)$.
	
	\begin{lemma}\label{lem:equi_critical}
		Given $m\in \N$ and $(p,q, Q_1)\in \cI_m$, there exist
		$\delta>0, \eps_0>0$ sufficiently small such that 
		for any $\eps\in (0, \eps_0)$ 
		if 	$\xi:=(\xi^0,Q_1)\in M_{\sigma,\xi^*}$ and $\xi^0$ is a critical point of $\xi^0 \mapsto \tilde{E}_{\eps}(\xi^0)$ if and only if     $u := W+\phi_{\xi,\eps}$ constructed by Theorem~\ref{thm2_p}  is a solution of~\eqref{eq:main_L2}.

	\end{lemma}
	\begin{proof}
		Suppose that $\xi^0:=(\xi_1,\ldots, \xi_{p+q})$ is a critical point of $\tilde{E}_{\eps}$. Then, for $i=1,\ldots,p+q, j=1,\ldots, \ii(\xi_i)$, Theorem~\ref{thm2_p} and   Lemma~\ref{lem4_p}  together with   \eqref{eq:lem3.5-0}-\eqref{eq:lem3.5-2}, \eqref{eq:solution_inf_dim} and \eqref{eq:pz_pa_phi_l}  imply that 
		\begin{align*} 
			0&= \partial_{(\xi_i)_j} \tilde{E}_{\eps}(\xi^0)= \lan
			u -i^*(F(u)) , \partial_{(\xi_{i})_j} (W+\phi_{\xi,\eps}) \ran=
			\sum_{\mytop{i'=1,\ldots, p+q}{j'=1,\ldots,\ii(\xi_{i'})}}  c^{\eps}_{i', j'} \lan PZ^{j'}_{i'}, \partial_{(\xi_{i})_j} (W+\phi_{\xi,\eps})\ran  \\
			&=c^{\eps}_{i, j} \lam_i \lan PZ^{j}_{i},PZ^j_i\ran +o\Big( \lam_i  \sum_{\mytop{i'=1,\ldots, p+q}{j'=1,\ldots,\ii(\xi_{i'})}}   |c^{\eps}_{i',j'}| \Big)
			=\frac{ 4\varrho(\xi_i)\lam_i }{3}c^{\eps}_{i,j}+o\Big( \lam_i  \sum_{\mytop{i'=1,\ldots, p+q}{j'=1,\ldots,\ii(\xi_{i'})}}   |c^{\eps}_{i',j'}| \Big),
		\end{align*}
		where $F(u)$ is defined by \eqref{eq:def_F}. 
		To sum up the estimates above for $i=1,\ldots,p+q$ and $j=1,\ldots,\ii(\xi_i)$, we deduce that 
		$\eps\rightarrow 0$
		\[ \sum_{\mytop{i=1,\ldots, p+q}{j=1,\ldots,\ii(\xi_{i})}}  \frac{ 4\varrho(\xi_i)\lam_i }{3} |c^{\eps}_{i,j}| \left(1+o(1) \right)=0. \]
		Then, there exists 
		$\eps_0>0$ sufficiently small  such that for $\eps\in (0,\eps_0)$,  we have 
		$c^{\eps}_{i,j}=0$ 
		for  all $i=1,\ldots,p+q, j=1,\ldots,\ii(\xi_i).$
		By~\eqref{eq:solution_inf_dim}, $u$ solves \eqref{eq:main_L2}. 
		
		Conversely, assume that
		$u:= W+\phi_{\xi,\eps} $, constructed in Theorem~\ref{thm2_p}, is in fact a  solution of~\eqref{eq:main_L2}. Then, by definition, 
		$\nabla E_\eps(u)=0$. Using $ \partial_{(\xi_i)_j}(W+\phi_{\xi,\eps}) $ as a test function, we deduce that 
		\begin{align*}
			0 &=\la \nabla E_\eps(u),  \partial_{(\xi_i)_j}(W+\phi_{\xi,\eps}) \ra= \int_\Sigma \Big(-\Delta_g u - F(u) +\overline{F(u)}\Big)  \partial_{(\xi_i)_j}(W+\phi_{\xi,\eps}) \, \d v_g \\
			&=     \la
			u -i^*(F(u)) , \partial_{(\xi_{i})_j} (W+\phi_{\xi,\eps})   \ra =\partial_{(\xi_{i})_j} \tilde{E}_\eps(\xi^0). 
		\end{align*}
	\end{proof}
	
	\begin{proof}[Proof of Theorem~\ref{thm:main}]
		Suppose that  $\xi^{0,*}$ is a $C^1$-stable critical point of $\ff$.  Denote that $\xi^*= (\xi^{0,*}, Q_1)=(\xi_1^*,\ldots, \xi^*_{p+q+N})$.  Then, there exists a sequence of $\xi^\eps=(\xi^{\eps}_1,\ldots,\xi^\eps_{p+q}, Q_1)\in M_{\sigma,\xi^*}$ with $d_g(\xi_i, \xi^*_i)<\eps$ for $i=1,\ldots, p+q$ such that  $ (\xi^{\eps}_1,\ldots,\xi^\eps_{p+q}) $ is a critical point of $\tilde{E}_\eps$.  
		
		Using Theorem \ref{thm2_p}, we have 
		$u_\eps:= \sum_{i=1}^{p+q+N} P\delta^{\gam_i}_{\lam_i, \xi^\eps_i}+\phi_{\xi^\eps,\eps}$, where $\lam_i$ is given by \eqref{def:lam_i} and $\|\phi_{\xi^\eps,\eps}\|\to 0$ as $\eps\to 0$.  Then, Lemma~\ref{lem:equi_critical} yields that 
		$u_{\eps}$ is a solution of  \eqref{eq:main_L2} with parameter $\eps\to 0$.  
		Denote  $\rho^\eps := \eps^2  \int_{\Sigma} K e^{u_{\eps}} \, \d v_g.$
		By the same approach in the proof of Theorem \ref{thm:main0}, we obtain that 
		\begin{align*}
			\rho^\eps&
			=\sum_{i=1}^{p+q+N}(1+\gam_i) \varrho(\xi_i) +o(1)\to 4\pi n_m=\rho_*; \\
			\rho^\eps\frac{ Ve^{v_{\eps}}} { \int_{\Sigma} Ve^{v_{\eps}}  \, \d v_g} \, \d v_g  &\stackrel{*}{\rightharpoonup}
			\sum_{i=1}^{p+q+N} (1+\gam(\xi^*_i))\varrho(\xi^*_i)\delta_{\xi^*_i},
		\end{align*}
		where $v_\eps= u_\eps-h_Q(x).$
		Consequently, $u_\eps$ is a sequence of solutions of \eqref{eq:mf} with parameter $\rho^\eps\to \sum_{i=1}^{p+q+N} (1+\gam_i(\xi^*_i))\varrho(\xi^*_i)=4\pi n_m$, which blows up at $\xi_1^*,\ldots, \xi_{p+q+N}^*$. 
		
		From the proof of  Lemma \ref{lem:key_energy_app} and Theorem~\ref{thm:expansion_E_reduced},  we have the following estimate for the energy level of the blow-up solutions $u_\eps$,  as $\eps\to 0$
		\begin{align*}
			J_{\rho^\eps}(u_\eps)
			&= (1+3\ln  2)\rho_*+\sum_{i=1}^{p+q+N} 2 (1+\gam_i)\varrho(\xi_i) \ln (1+\gam_i)-\frac 1 2 \ff(\xi^\eps_1,\ldots,\xi^\eps_{p+q})\\
			&\quad  -\rho_* \ln  \rho_*+o(1)\\
			&\to\rho_* -\rho_* \ln  \Big(\frac {\rho_*}{8}\Big)+ 
			\sum_{i=1}^{p+q+N} 2 (1+\gam_i)\varrho(\xi_i) \ln (1+\gam_i)-\frac 1 2 \ff(\xi^{0,*}).
		\end{align*}

	\end{proof}

	\appendix
	\section{The linearized equations}\label{app:a}
	\begin{lemma}\label{lem:profile_lin}
		Let { $\gam \in (-1,+\infty)\backslash \N_0$,} and  $\phi$ be a distributional solution of 
		$$
		\left\{\begin{array}{ll}
			-\Delta \phi(y)=|y|^{2 \gam} \frac{ 8 (1+\gam)^2 }{( 1+ |y|^{2(1+\gam)})^2}\phi(y), \quad &\text { in } \cR, \\
			\partial_{y_2} \phi(y)=0, &\text{ on } \partial\cR, 
		\end{array}\right.
		$$
		with the restriction $\int_{\cR} |\nabla \phi|^2<\infty,$	where  $\cR=\R^2 \text{ or }\R^2_+ $. 
		Then, there exists $c_0\in \R$ such that $ \phi= c_0\frac{1-|y|^{2(1+\gam)}}{1+|y|^{2(1+\gam)}} .$
	\end{lemma}
	\begin{proof}
		For the case $\cR = \R^2$, we refer to  \cite[Lemma 2.1]{BartolucciYangZhang2024} and \cite[Lemma 2.1]{Zhang2009}. 
		For the case $\cR = \R_{+}^2$, we extend $\phi$ to a function defined on $\R^2$ by even reflection across the $y_1$-axis, that is,
	$$
		\phi(y_1, y_2) = \begin{cases}
			\phi(y_1, y_2) & \text{for } y_2 \geq 0, \\
			\phi(y_1, -y_2) & \text{for } y_2 < 0. 
		\end{cases}
	$$
	Since $\partial_{y_2}\phi \equiv 0$ on $\partial \R^2_+$, it follows that
	\[
	|y|^{2\gamma} e^{\tilde{\delta}^{\gamma}} \phi(y) \in L^r(\R^2)
	\quad \text{for some } r>1 .
	\]
	By elliptic regularity, the even extension of $\phi$ across $\partial \R^2_+$
	belongs to $C^{1,s}_{\mathrm{loc}}(\R^2)$ for some $s\in(0,1)$.
	Moreover, standard elliptic estimates yield
	\[
	\phi \in C^2_{\mathrm{loc}}(\R^2\setminus\{0\}),
	\]
	and $\phi$ is bounded on every compact subset of $\R^2$.
	Consequently, Lemma~\ref{lem:profile_lin} follows exactly as in the case
	$\cR=\R^2$.
	
	We emphasize that when $\gamma\in(-1,0)$, a global sublinear growth condition of the form
	\[
	|\phi(y)|\le C(1+|y|)^t, \text{ for some } t\in (0,1)
	\]
	does not hold in general.
	Nevertheless, the weaker assumption
	\[
	\int_{\R^2} |\phi|^2\,\d y<+\infty
	\]
	is sufficient to derive the required conclusion via Fourier analysis, as carried out in
	\cite{BartolucciYangZhang2024}.
	We omit the details here.

	\end{proof}

	
	\section{Asymptotic behavior of the projected bubbles}\label{app:b}
	In this section, we first obtain estimates for the projected bubbles and their derivatives, and then use these estimates to examine the basic asymptotic behavior of the approximate solution.

	We assume that $(p,q, Q_1)\in \N\times\N \times\cP(Q)$ with $Q_1=\{  \xi^*_{p+q+1}, \ldots,\xi^*_{p+q+N} \}$ and $P\delta_i=P\delta^{\gam_i}_{\lam_i,\xi_i}$ with $\lam_i^{1+\gam_i}\asymp \frac 1 \eps$  for $i=1,\ldots, p+q+N$ as $\lam_i  \to +\infty$.  
	Let $\xi^*=(\xi^{0,*}, Q_1)$ and $\xi^{0,*}$ is a critical point of $\ff$ defined by \eqref{def:reduce_fun_0}. Assume that $$\xi:=(\xi_1, \ldots, \xi_{p+q+N})=(\xi_1,\ldots, \xi_{p+q}, \xi^*_{p+q+1}, \ldots,\xi^*_{p+q+N})\in M_{\sigma, \xi^*}$$ 
	for some $\sigma>0$ sufficiently small.  
	
	
	We also denote  for $i=1,\ldots, p+q+N$
	\begin{equation}
		\label{def:varep_0i} \epsilon_{0,i}=\begin{cases}
			\frac {\varrho(\xi_i)\mathrm{c}(\gam_i) }4 \lam_i^{-2}   & \gam_i> 0\\
			\frac {\rho(\xi_i)} 2 	\lam_i^{-2} \ln \lam_i & \gam_i=0\\
			0& \gam_i \in (-1, 0)
		\end{cases},  \quad \text{ and } \quad 
		\epsilon_{1,i}= \begin{cases}
			\lam_i^{-2(1+\gam_i)}\ln \lam_i &   \gam_i>0  \\
			\lam_i^{-2}  &\gam_i=0\\
			\lam_i^{-2(1+\gam_i)}	 &\gam_i\in (-1, 0)  
		\end{cases}.
	\end{equation}
	where  $\cc(\gam):= \int_{0}^{\infty}  \ln ( 1+ s^{-(1+\gam)}) \, \d s=\frac{\pi }{1+\gam}\csc (\frac{\pi}{1+\gam})<+\infty$ for any $\gam>0$.
	
	\begin{lemma}\label{lem:proj_bubble}
		For any $i=1,\ldots, p+q+N$, as $\lam_i  \to\infty$ we have 
		\begin{align*}
			P\delta^{\gam_i}_{\lam_i,\xi_i}&=\chi_i \cdot(\delta^{\gam_i}_{\lam_i,\xi_i}-\ln  (8 (1+\gam_i)^2 \lam_i^{- 2(1+\gam_i)}))+(1+\gam_i) \varrho(\xi_i)H^g(\cdot, \xi_i)+\epsilon_{0,i}
			+\cO(\epsilon_{1,i})	
		\end{align*}
		in $C(\Sigma)$, where $\epsilon_{0,i}$ and $\epsilon_{1,i}$ are given by \eqref{def:varep_0i}. 
	\end{lemma}
	
	\begin{proof}
		The proof for the case $\gam_i = 0$ is given in \cite[Appendix A]{ABH2024MF}, and the case $\gam_i \neq 0$ can be treated in a similar manner. The argument relies on the regularity theory for elliptic equations with Neumann boundary conditions (see, e.g., \cite{N2014, A1959, W2004}), followed by  the Sobolev embedding. The details are omitted here.
	\end{proof}
	The following lemma shows the asymptotic expansions of the first derivatives of the projected bubbles  along $\lam_i\asymp \eps^{-\frac{1}{1+\gam_i}}\to\infty$:
	\begin{lemma} 
		\label{lem:proj_1st}
		We  have as $\lam_i \to\infty$,  for $i=1,\ldots, p+q+N$
		\begin{align*}
			\lam_i \partial_{\lam_i} P\delta^{\gam_i}_{\lam_i,\xi_i} 
			&=\chi_i \frac{4(1+\gam_i)}{1+\lam_i^{2(1+\gam_i)} |y_{\xi_i}|^{2(1+\gam_i)}}+\cO(\epsilon_{0,i}+\epsilon_{1,i}), 
		\end{align*}
		and for $i=1,\ldots, p+q$ and $j=1,\ldots, \ii(\xi_i)$
		\begin{align*}
			&\quad	\frac {\partial_{x_j} P\delta^{\gam_i}_{\lam_i,x}|_{x=\xi_i}} {\lam_i}  =\frac {\chi_i \partial_{x_j}  \delta^{\gam_i}_{\lam_i,x}|_{x=\xi_i} } {\lam_i} +   \frac {\varrho(\xi_i)} {\lam_i}\partial_{(\xi_i)_j} H^g(\cdot,\xi_i)-\frac 4 {\lam_i} (\ln  |y_{\xi_i}|)\partial_{(\xi_i)_j}\chi_{i} +\cO(\epsilon_{0,i}+	\epsilon_{1,i}) , 
		\end{align*}
		in $C(\Sigma)$, where $\epsilon_{1,i}$ is defined by \eqref{def:varep_0i}. 
		
	\end{lemma}
	\begin{proof}
		The proof for the case $\gam_i = 0$ is given in \cite[Appendix A]{ABH2024MF}, and the case $\gam_i \neq 0$ can be deduced similarly.   The details are omitted here. 
	\end{proof}
	\begin{remark}\label{rk:B.1}
		Recall that $PZ^j_i$ is defined by \eqref{def:PZ}.  Observe that 
		\[ \lam_i\partial_{\lam_i} P\delta_i =PZ^0_i \quad  \text{ and }\quad 
		\lam^{-1}_i\partial_{(\xi_i)_j} P\delta_i  	=PZ^j_i+ \cO( \lam_i^{-2}\ln \lam_i ), \]
		for $i=1,\ldots, p+q+N, j=1,\ldots, \ii(\xi_i). $
		Then,  we also have the following asymptotic estimates by  similar arguments: 
		for $i=1,\ldots, p+q+N$
		$$PZ^0_i=  \frac{4(1+\gam_i)\chi_i}{1+\lam_i^{2(1+\gam_i)}|y_{\xi_i}|^{2(1+\gam_i)}} + \begin{cases}
			\cO(	\lam_i^{-2}) &   \gam_i>0  \\
			\cO(	\lam_i^{-2} \ln \lam_i ) &\gam_i=0\\
			\cO(	\lam_i^{-2(1+\gam_i)}	)  &\gam_i\in (-1, 0)  
		\end{cases}, $$ in $C(\Sigma);$ for $i=1,\ldots,p+q, j=1,\ldots,\ii(\xi_i)$, 
		$PZ^j_i= \frac{ 4 \chi_i \lam_i (y_{\xi_i})_j } 
		{ 1+ \lam_i^{2(1+\gam_i) } |y_{\xi_i}|^{2(1+\gam_i)} }+\cO(\lam_i^{-1}), $
		in $C(\Sigma)$. 
	\end{remark}
	
	Next, we will show the asymptotic orthogonal property of the projected bubbles. 
	\begin{lemma}
		\label{lem:innner_proj}
		For $i, j=1,\ldots, p+q+N$, we have 
		\begin{align} \label{eq:B3-1}
			&\quad 	\la P\idelta, P\delta^{\gam_j}_{\lam_j, \xi_j}\ra  \\
			&= 	(-2(1+\gam_i)\varrho(\xi_i)+ 4(1+\gam_i)^2\varrho(\xi_i)\ln \lam_i) \bdelta_{ij}+ (1+\gam_i)(1+\gam_j ) \varrho(\xi_i)\varrho(\xi_j) H^g(\xi_i,\xi_j)\nonumber \\
			&\quad   + \big( (1+\gam_i) \varrho(\xi_i) \epsilon_{0,j} +     (1+\gam_j) \varrho(\xi_j) \epsilon_{0,i} \big) +\cO(\epsilon_{1,i}+\epsilon_{1,j}), \nonumber
		\end{align}
		as $\lam_i \asymp \lam^{\frac 1 {1+\gam_i}}\to \infty,$ 	where  $\bdelta_{ij}$ is the Kronecker delta. 
	\end{lemma}
	\begin{proof}
		Using the expansions in Lemma~\ref{lem:proj_bubble} again, we can deduce that for any  $i,j=1, \ldots, p+q+N$  
		\begin{align*}
			\la P\idelta, P\delta^{\gam_j}_{\lam_j, \xi_j}\ra 
			&=\int_\Sigma \chi_i |y_{\xi_i}|^{2\gam_i} e^{-\varphi_i} e^{\idelta}\big (\chi_j \cdot(\delta^{\gam_j}_{\lam_j,\xi_j}-\ln  (8 (1+\gam_j)^2 \lam_j^{- 2(1+\gam_j)}))\\
			&\quad +(1+\gam_j) \varrho(\xi_j)H^g(\cdot, \xi_j)
			+ \epsilon_{0,j}
			+\cO( \epsilon_{1,j}  )\big) \, \d v_g\\
			&=\int_\Sigma \chi_i |y_{\xi_i}|^{2\gam_i} e^{-\varphi_i} e^{\idelta}  
			\chi_j \ln \frac{ \lam_j^{4(1+\gam_j)} |y_{\xi_j}|^ {4(1+\gam_j)}}{( 1+ \lam_j^{2(1+\gam_j)} |y_{\xi_j}|^ {2(1+\gam_j)})^2} \, \d v_g \\
			&\quad+(1+\gam_j)\varrho(\xi_j) \int_\Sigma ( -\Delta _g P\delta^{\gam_i}_{\lam_i,\xi_i}) G^g(\cdot,\xi_j)  \, \d v_g 
			+ (1+\gam_i)\varrho(\xi_i) \epsilon_{0,j}+\cO(\epsilon_{1,j})\\
			&=  -2(1+\gam_i)\varrho(\xi_i)\bdelta_{ij} + (1+\gam_j) \varrho(\xi_j) P\delta^{\gam_i}_{\lam_i,\xi_i}(\xi_j)  + (1+\gam_i)\varrho(\xi_i)\epsilon_{0,j}+\cO(\epsilon_{1,j}), 
		\end{align*}
where we applied  the Green's representation formula for $P\delta^{\gam_i}_{\lam_i,\xi_i} $ and 
		 the following estimates: 
		\begin{align*}
			&\quad 	\int_\Sigma \chi_i |y_{\xi_i}|^{2\gam_i} e^{-\varphi_i} e^{\idelta} \, \d v_g 
			= (1+\gam_i) \varrho(\xi_i) +\cO( \lam_i^{-2(1+\gam_i)} ) , \\
			&\quad 2\int_\Sigma \chi_i |y_{\xi_i}|^{2\gam_i} e^{-\varphi_i} e^{\idelta}\ln   \frac{ \lam_i^{2(1+\gam_i)}|y_{\xi_i}|^{2(1+\gam_i)} }{1+ \lam_i^{2(1+\gam_i)}|y_{\xi_i}|^{2(1+\gam_i)}}\, \d v_g = -2(1+\gam_i) \varrho(\xi_i) +\cO( \lam_i^{-2(1+\gam_i)} ). 
		\end{align*}
		Then, Lemma~\ref{lem:proj_bubble}, together with the estimates above, yields \eqref{eq:B3-1}.
	\end{proof}
	Next, we show the orthogonal properties for $PZ^j_i$ for $i=1,\ldots, p+q, j=1,\ldots, \ii(\xi_i)$ and $PZ^0_i$ for $i=1,\ldots,p+q+N$. 
	\begin{lemma}~\label{lem4_p}
		As $\eps\rightarrow 0$,  we have the following asymptotic estimates: 
		\begin{align*}
			&\quad	\langle PZ^j_i, PZ^{j'}_{i'}\rangle=\begin{cases}
				\frac 4 3 (1+\gam_i)^2 \varrho(\xi_i)	\bdelta_{i'i}\bdelta_{j'j} +\cO(  \epsilon_{0,i} ) \quad \text{ when } j \text{ or } j'=0, \\
				\frac{4}{3} \varrho(\xi_i)   \bdelta_{i'i}\bdelta_{j'j} +\cO( \lam_i^{-1} )  \quad  \text{ when } i, i'=1,\ldots, p+q, j=1,\ldots, \ii(\xi_i), 
			\end{cases}
		\end{align*}
		where  $\bdelta_{ij}$ is the Kronecker delta and $\epsilon_{0,i}$ is give by \eqref{def:varep_0i}. 
	\end{lemma}
	\begin{proof}
		We begin the proof by considering the case $j'=0$. 
		Using Remark \ref{rk:B.1}, we have 
		\begin{align*}
			\langle PZ^j_i, PZ^{j'}_{i'}\rangle &= \int_\Sigma (-\chi_i \Delta_g Z^j_i) ( \chi_{i'} (Z^0_{i'}+2) +\cO( \epsilon_{0,i'})) \, \d v_g \\
			&= -\int_{\Sigma} \chi_i |y_{\xi_i}|^{2\gam_i}  e^{-\varphi_i}e^{\delta_i } Z^j_i (Z^0_{i'}+2) +\cO( \epsilon_{0,i'}) \, \d v_g  \\
			& =    \frac{ 4}{3} (1+\gam_i)^2  \varrho(\xi_i)\bdelta_{i'i} \bdelta_{0j} +\cO( \epsilon_{0,i'}), 
		\end{align*}
		where  we used the integral  $\int_{\R^2} \frac{64 (1+\gam)^4 ( 1-|y|^{2(1+\gam)}) |y|^{2\gam}}{(1+|y|^{2(1+\gam)})^ 4 } \, \d y= \frac {32\pi} 3  (1+\gam)^3$ for any $\gam>-1.$
		It remains to consider the case $j,j'\neq 0$. Then, it follows that $i,i'\in \{1,\ldots, p+q\}$. 
		By Remark \ref{rk:B.1} again, we have 
		\begin{align*}
			\langle PZ^j_i, PZ^{j'}_{i'}\rangle &=  \int_\Sigma (-\chi_i \Delta_g Z^j_i) ( \chi_{i'} Z^{j'}_{i'} +\cO(  \lam_{i'}^{-1}) )\, \d v_g = \frac  4  3 \varrho(\xi_i) \bdelta_{ii'}\bdelta_{jj'}  +\cO( \lam_{i' }^{-1}), 
		\end{align*}
		where we applied that $ \int_{\R^2} \frac{ 2 y_j^2}{(1+|y|^2)^4} \, \d y= \int_{\R^2} \frac{|y|^2}{(1+|y|^2)^4} \, \d y =\frac{\pi}{6}.$
	\end{proof}

	\begin{lemma}
		\label{lem:diif_e^W_sum_e^u}   For any $\iota\in (0, 1)$, there exists $s_0>0$ sufficiently small such that 
		for  any 
		$s\in (1, s_0)$, 	
		$$\Big  \| \eps^2  Ke^{\sum_{i=1}^{p+q+N} P\delta_i}  -\sum_{i=1}^{p+q+N} \chi_i |y_{\xi_i}|^{2\gam_i} e^{-\varphi_i}e^{\delta_i} \Big \|_s  =\mathcal{O}(\eps^{ \frac{\iota}{1+\max\{0, \gam_*\}}} ),$$
		as $\eps \rightarrow 0,$
		where $\gam_*=\max_{i=1,\ldots, p+q+N}\gam_i$. 
	\end{lemma}
	\begin{proof}
		{\sc Case I}.  $\gam_*>0$.
		For  $i=1,\ldots, p+q+N$ and  $y\in \Omega_i= \lam_i  B^{\xi_i}_{2r_0}$, we have   
		\begin{align*}
			\Theta_{i}(y) &:= ( W- \delta_i +\varphi_i + 2\ln  \eps +\ln  K_i) \circ y^{-1}_{\xi_i}( \frac y {\lam_i})\\
			&= - \ln  ( 8(1+\gam_i)^2) + (1+\gam_i) \varrho(\xi_i)R^g(\xi_i) +\sum_{\mytop{j=1,\ldots, p+q+N}{j\neq i}}(1+\gam_j)\varrho(\xi_j) G^g(\xi_i,\xi_j) + \ln  K_i(\xi_i) \\
			& \quad  +2(1+\gam_i) \ln  \lam_i  +2\ln  \eps
			+ \sum_{j=1}^{p+q+N} \epsilon_{0,j}+ \cO\Big( \sum_{j=1}^{p+q+N}  (\epsilon_{1,j} + \lam_i^{-1}|y|) \Big) \\
			&=  \sum_{j=1}^{p+q+N} \epsilon_{0,j}+ \cO\Big( \sum_{j=1}^{p+q+N}  (\epsilon_{1,j} + \lam_i^{-1}|y|) \Big), 
		\end{align*}
		in view of \eqref{def:lam_i} and \eqref{def:d_i}. 
		Using Lemma~\ref{lem:proj_bubble} and \eqref{def:lam_i}, we have   
		\begin{align*}
			&\quad \int_{\Sigma} \Big|  \eps^2  Ke^W  -\sum_{i=1}^{p+q+N} \chi_i |y_{\xi_i}|^{2\gam_i} e^{-\varphi_i}e^{\delta_i} \Big|^s \,\d v_g 
			\\
			&=\sum_{i=1}^{p+q+N}  \lam_i^{-2(1+\gam_i s)} \int_{\lam_i B^{\xi_i}_{r_0}(\xi_i)} |y|^{2\gam_i s}e^{(1-s) \varphi_{\xi_i}\circ y_{\xi_i}^{-1}(\frac y {\lam_i} )}\Big|  \eps^2  K_i \circ y_{\xi_i}^{-1}(\frac y {\lam_i}) e^{\varphi_{\xi_i}\circ y_{\xi_i}^{-1}(\frac y{\lam_i})}  e^{W\circ y_{\xi_i}^{-1}(\frac y{\lam_i})}\\
			& \quad -  e^{\delta_i\circ y_{\xi_i}^{-1}(\frac y{\lam_i})}  \Big|^s \, d y+ \mathcal{O}(\eps^{2s}) \\
			&=\cO\Big(  \sum_{i=1}^{p+q+N}  \lam_i^{-2+2s}\int_{\Omega_i} \frac{|y|^{2\gam_i s}}{(1+|y|^{2(1+\gam_i)})^{2s}}|e^{\Theta_{i}(y)}-1 |^s \, \d y \Big) + \mathcal{O}(\eps^{2s}) \\
			&=\cO( \eps^{s\frac{2}{1+\gam_*}+ \frac{s}{1+\gam_-} (\frac{2}{s}-2)}+ \eps^{s\frac{1}{1+\gam_*}+ \frac{s}{1+\gam_*}(\frac{2}{s}-2)} )=\cO( \eps^{s\frac{\iota}{1+\gam_*}} ),  
		\end{align*}
		for any $s\in (1, \min\{\frac{ 2}{1+\iota}, (1- \frac{(1-\iota)(1+\gam_-)}{2(1+\gam_*)} )^{-1}\})$, where $\gam_-=\min\{\gam_i: i\in 1, \ldots, p+q+N\}. $
		
		{\sc Case II. } $\gam_*= 0$.
		Similarly, in this case we have 
		$
		\Theta_{i}(y) 
		=   \cO( \eps^2|\ln \eps|+ \frac {|y|} {\lam_i}). 
		$
		Then, it follows that 
		\begin{align*}
			&\quad \int_{\Sigma} \Big|  \eps^2  Ke^W  -\sum_{i=1}^{p+q+N} \chi_i |y_{\xi_i}|^{2\gam_i} e^{-\varphi_i}e^{\delta_i} \Big|^s \,\d v_g 
			\\
			&=\cO\Big(  \sum_{i=1}^{p+q+N}  \lam_i^{-2+2s}\int_{\Omega_i} \frac{|y|^{2\gam_i s}}{(1+|y|^{2(1+\gam_i)})^{2s}}|e^{\Theta_{i}(y)}-1 |^s \, \d y \Big) + \mathcal{O}\Big(\eps^{2s}+\sum_{i=1}^{p+q+N}\lam_i^{-2(1+\gam_i)s}\Big) \\
			&=\cO(\eps^{ 2s + \frac{s}{1+\gam_-} ( \frac 2 s -2 )}|\ln \eps|  + \eps^{s + s( \frac 2 s -2) }                                      )=\cO( \eps^{s \iota}),
		\end{align*}
		where 
		$s\in (1,  \min\{\frac{ 2}{1+\iota}, (1- \frac{(1-\iota)(1+\gam_-)}{2(1+\gam_*)} )^{-1}\} )$.

		{\sc Case III.  $\gam_*<0$. }
		Similarly, in  this case we have 
		$
		\Theta_{i}(y) 
		=   \cO(\eps^2+ \frac {|y|} {\lam_i}). 
		$
		Then, it follows that 
		\begin{align*}
			&\quad \int_{\Sigma} \Big|  \eps^2  Ke^W  -\sum_{i=1}^{p+q+N} \chi_i |y_{\xi_i}|^{2\gam_i} e^{-\varphi_i}e^{\delta_i} \Big|^s \,\d v_g 
			\\
			&=\cO\Big(  \sum_{i=1}^{p+q+N}  \lam_i^{-2+2s}\int_{\Omega_i} \frac{|y|^{2\gam_i s}}{(1+|y|^{2(1+\gam_i)})^{2s}}|e^{\Theta_{i}(y)}-1 |^s \, \d y \Big) + \mathcal{O}\Big(\eps^{2s}+\sum_{i=1}^{p+q+N}\lam_i^{-2(1+\gam_i)s}\Big) \\
			&=\cO( \eps^{ 2 s+ \frac{s}{(1+\gam_-)} ( \frac 2 s - 2 )}+ \eps^{\frac s  {1+\gam_*}+ \frac{s}{1+\gam_*}(\frac 2 s -2)} )=\cO( \eps^{s\iota} ), 
		\end{align*}
		where 
		$s\in (1, \min\{\frac{ 2}{1+\iota}, (1- \frac{(1-\iota)(1+\gam_-)}{2} )^{-1}\} ) )$.
		
	\end{proof}

	
	

	\begin{flushleft}
		\noindent {\sc Mohameden Ahmedou}\\
		
		\noindent  {\it Mathematisches Institut der Justus-Liebig-Universit\"at Giessen} \\ 
		Arndtsrasse 2, D-35392 Giessen, Germany \\
		\textsf{\href{mailto:Mohameden.Ahmedou@math.uni-giessen.de}{Mohameden.Ahmedou@math.uni-giessen.de}}
		
		\noindent
		{\sc Zhengni Hu}\\
		{\it School of Mathematical Sciences,
			Shanghai Jiao Tong University}\\
		800 Dongchuan RD, Minhang District, 200240 Shanghai, China\\
		\textsf{\href{mailto:zhengni_hu@sjtu.edu.cn}{zhengni\_hu@sjtu.edu.cn}}\\
		\vskip10pt
		
		\noindent {\sc Miaomiao Zhu} \\
		{\it School of Mathematical Sciences, Shanghai Jiao Tong University}\\
		800 Dongchuan Road, Shanghai, 200240, P. R. China \\ 
		\textsf{\href{mailto: mizhu@sjtu.edu.cn}{mizhu@sjtu.edu.cn}
		}
		
	\end{flushleft}

\end{document}